\documentclass[11pt]{article}
    \usepackage{url}
    \usepackage{verbatim}
    \usepackage[titletoc]{appendix}
    \usepackage{graphicx}
	\usepackage{amsmath}
	\usepackage{amsthm}
	\usepackage{amsfonts}
	\usepackage{graphicx}
    \usepackage{nicefrac}
    \usepackage{longtable}
    \usepackage{color}
    \usepackage{graphicx, amssymb, subcaption,graphics}

    \newcommand{\be}{\begin{equation}}
    \newcommand{\ee}{\end{equation}}

    \newcommand{\nrm}[1]{\left\| #1 \right\|}

    \def\f{\mbox{\boldmath $f$}}
    
    \def\v{\mbox{\boldmath $v$}}
    \def\w{\mbox{\boldmath $w$}}
    \newcommand\dt {{\Delta t}}
     \def\0{\mbox{\boldmath $0$}}

	\newtheorem{thm}{Theorem}[section]
	\newtheorem{prop}[thm]{Proposition}
	\newtheorem{cor}[thm]{Corollary}
	
	\newtheorem{lem}[thm]{Lemma}
	\newtheorem{rem}[thm]{Remark}

\begin{document}
	\title{A third order BDF energy stable linear scheme for the no-slope-selection thin film model}
	
	\author{
Yonghong Hao \thanks{College of Applied Sciences, Beijing University of Technology, Beijing 100124, P. R. China (15933223502@163.com)}
\and
Qiumei Huang \thanks{College of Applied Sciences, Beijing University of Technology, Beijing 100124, P. R. China (qmhuang@bjut.edu.cn)}
\and		
Cheng Wang\thanks{Department of Mathematics, The University of Massachusetts, North Dartmouth, MA  02747, USA (Corresponding Author: cwang1@umassd.edu)}
}

	\maketitle
	\numberwithin{equation}{section}

	\begin{abstract}
In this paper we propose and analyze a (temporally) third order accurate backward differentiation formula (BDF) numerical scheme for the no-slope-selection (NSS) equation of the epitaxial thin film growth model, with Fourier pseudo-spectral discretization in space. The surface diffusion term is treated implicitly, while the nonlinear chemical potential is approximated by a third order explicit extrapolation formula for the sake of solvability. In addition, a third order accurate Douglas-Dupont regularization term, in the form of $-A \dt^2 \Delta_N^2 ( u^{n+1} - u^n)$, is added in the numerical scheme. A careful energy stability estimate, combined with Fourier eigenvalue analysis, results in the energy stability in a modified version, and a theoretical justification of the coefficient $A$ becomes available. As a result of this energy stability analysis, a uniform in time bound of the numerical energy is obtained. And also, the optimal rate convergence analysis and error estimate are derived in details, in the $\ell^\infty (0,T; \ell^2) \cap \ell^2 (0,T; H_h^2)$ norm, with the help of a linearized estimate for the nonlinear error terms. 
Some numerical simulation results are presented to demonstrate the efficiency of the numerical scheme and the third order convergence. The long time simulation results for $\varepsilon=0.02$ (up to $T=3 \times 10^5$) have indicated a logarithm law for the energy decay, as well as the power laws for growth of the surface roughness and the mound width. In particular, the power index for the surface roughness and the mound width growth, created by the third order numerical scheme, is more accurate than those produced by certain second order energy stable schemes in the existing literature.
	\end{abstract}
	
	
\medskip
	

\section{Introduction}
In this article we consider a no-slope-selection (NSS) epitaxial thin film growth equation, which corresponds to the $L^2$ gradient flow associated with the following energy functional
\begin{equation}\label{energy-NSS}
E(u) := \int_{\Omega}\left(- \frac{1}{2} \ln(1+|\nabla u|^2) +\frac{\varepsilon^2}{2} |\Delta u|^2\right)\mbox{d} \textbf{x},
\end{equation}
where $\Omega = (0, L_x) \times (0, L_y)$, $u:\Omega\rightarrow \mathbb{R}$ is a periodic height function, and  $\varepsilon$ is a constant parameter of transition layer width. In more details, the first nonlinear term represents the Ehrlich-Schwoebel (ES) effect~\cite{Ehrlich1966, libo06, libo03, libo04, Schwoebel1969}, 
which results in an uphill atom current in the dynamics and the steepening of mounds in the film.  The second higher order quadratic term represents the isotropic surface diffusion effect~\cite{libo03, moldovan00}. In turn, the chemical potential becomes the following variational derivative of the energy
\begin{equation}  \label{NSS-chem pot}
 \mu := \delta_{u}E = \nabla\cdot\left(\frac{\nabla u}{1+|\nabla u|^2}\right) + \varepsilon^2 \Delta^2 u,
\end{equation}
and the PDE stands for the $L^2$ gradient flow
\begin{equation}\label{equation-NSS}
\partial_t u = - \mu = -\nabla\cdot\left(\frac{\nabla u}{1+|\nabla u|^2}\right) - \varepsilon^2 \Delta^2 u .
\end{equation}
Meanwhile, under a small-slope assumption that $|\nabla u|^2\ll 1$, the energy functional could be approximated by a polynomial pattern
\begin{equation}  \label{energy-SS}
E(u) = \int_{\Omega}\left( \frac14 ( | \nabla u |^2 -1)^2 +\frac{\varepsilon^2}{2} |\Delta u|^2\right)\mbox{d} \textbf{x} ,
\end{equation}
and the dynamical equation is formulated as
	\begin{equation}
\partial_t u =  \nabla \cdot \left(|\nabla u|^2\nabla u \right) - \Delta u - \varepsilon^2 \Delta^2 u .
	\label{equation-SS}
	\end{equation}
This model is referred to as the slope-selection (SS) equation~\cite{kohn06, kohn03, libo03, moldovan00}.  A solution to (\ref{equation-SS}) exhibits pyramidal structures, where the faces of the pyramids have slopes $|\nabla u| \approx 1$; meanwhile, the no-slope-selection equation (\ref{equation-NSS}) exhibits mound-like structures, and the slopes of which (on an infinite domain) may grow unbounded~\cite{libo03, wang10}. On the other hand, both solutions have up-down symmetry in the sense that there is no way to distinguish a hill from a valley.  This can be altered by adding adsorption/desorption or other dynamics.



The numerical schemes with high order accuracy and energy stability have been of great interests, due to the long time nature of the gradient flow coarsening process. There have been many efforts to devise and analyze energy stable numerical schemes for both the SS and NSS equations; see the related references~\cite{chenwang12, chen14, Huang2019, Ju18, LiD2016a, Liao2020a, qiao12, qiao15, qiao17, qiao12b, shen12, wang10, WangS2020a, xu06, yang17b}, etc.  
In particular, the linear schemes have been attracted a great amount of attentions among the energy stable numerical approaches, due to its simplicity of implementation. 
In the existing works~\cite{chen12, LiW18}, the authors proposed linear algorithms for the NSS equation, with first and second order temporal accuracy orders, respectively, so that the energy stability could be established at a theoretical level. Such a highly nonlinear energy stability analysis is based on the following subtle fact: in spite of its complicated form in the denominator, the nonlinear term in the NSS equation~\eqref{equation-NSS} has automatically bounded higher order derivatives in the $L^\infty$ norm. Such an idea has also been applied to derive the exponential time differencing (ETD) based numerical schemes for the NSS equations, with certain modified energy stability; see the related works~\cite{chen20a, Ju18}, etc. 

  Among the energy stable numerical works for the epitaxial thin film model, the theoretical analysis of a (temporally) third order accurate numerical scheme is very limited. There have been two recent works~\cite{chen20b, cheng2019c} to address the energy stability of third order accurate schemes for the NSS equation, based on the ETD approach. Meanwhile, it is observed that, some non-trivial operators have been involved in the numerical implementation of the ETD-based higher oder schemes, which may lead to a high computational cost. In this article, we propose and analyze a third order accurate  numerical scheme for the NSS equation~\eqref{equation-NSS}, based on the standard BDF3 temporal approximation, combined with certain explicit extrapolation formula for the nonlinear term. Again, such an explicit treatment to the nonlinear terms would not be able to ensure the energy stability at the theoretical level. To overcome this difficulty, we have to add a third order Douglas-Dupont regularization term in the numerical scheme, namely in the form of $-A \dt^2  \Delta_N^2 ( u^{n+1} - u^n)$. Furthermore, a careful energy estimate enables us to derive a rigorous stability estimate for a modified energy function, which contains the original energy functional and a few non-negative numerical correction terms. In fact, the subtle fact that all the nonlinear terms have automatically bounded higher order derivative will play an important role in the highly complicated nonlinear analysis. The Fourier pseudo-spectral method is taken as the spatial approximation, and the discrete summation by parts property will facilitate the corresponding analysis for the fully discrete scheme. As a result of this modified energy stability, we are able to derive a uniform-in-time bound for the original energy functional bound.
In addition to the energy stability analysis, we provide a theoretical analysis of an $O (\dt^3 + h^m)$ rate convergence estimate for the proposed third order BDF (BDF3) scheme, in the $\ell^\infty(0,T; \ell^2) \cap \ell^2 (0, T; H_h^2)$ norm. It is well-known that a direct inner product with the numerical error equation by $e^{n+1}$ (the error function at time step $t^{n+1}$) does not lead to the desired result, because of the long stencil structure involved. Instead, an inner product with $e^{n+1} + (e^{n+1} - e^n)$ is considered in the analysis, originated from an existing work~\cite{JLiu2013}. In turn, the full order convergence result is expected via detailed numerical error estimates. Again, a uniform bound of the nonlinear derivatives will play an important role in the derivation of such an error estimate. And also, all the spatial operators are associated with the standard derivatives, so that a complicated eigenvalue analysis could be avoided in the derivation of the optimal convergence, in contrast with the ones reported in~\cite{chen20b, cheng2019c}.

The long time simulation results for the coarsening process have indicated a logarithm law for the energy decay, as well as the power laws for growth of the surface roughness and the mound width. In particular, the power index for the surface roughness and the mound width growth, created by the proposed third order BDF scheme, is more accurate than those created by certain second order schemes in the existing literature, with the same numerical resolution. This experiment has demonstrated the robustness of the proposed BDF3 numerical scheme.

The rest of the article is organized as follows. In Section~\ref{sec-num scheme} we present the numerical scheme, including the review of the Fourier pseudo-spectral spatial approximation. Afterward, a modified energy stability is established for the proposed third order BDF scheme.  Subsequently, the $\ell^\infty(0,T; \ell^2) \cap \ell^2 (0, T; H_h^2)$ convergence estimate is provided in Section~\ref{sec-convergence}. In Section~\ref{sec:numerical results} we present the numerical results, including the accuracy test and the long time simulation for the coarsening process.   Finally, the concluding remarks are given in Section~\ref{sec:conclusion}.


	\section{The numerical scheme}
	\label{sec-num scheme}
	
\subsection{Review of the Fourier pseudo-spectral approximation}

For simplicity of presentation, we assume that the domain is given by $\Omega = (0,L)^2$, $N_x = N_y = N$ and $N \cdot h = L$. A more general domain could be treated in a similar manner. Furthermore, to facilitate the pseudo-spectral analysis in later sections, we set $N = 2K+1$. All the variables are evaluated at the regular numerical grid $(x_i, y_j)$, with $x_i = i h$, $y_j=jh$, $0 \le i , j \le 2K +1$.

Without loss of generality, we assume that $L=1$. For a periodic function $f$ over the given 2-D numerical grid, set its discrete Fourier expansion as
\begin{equation}
  f_{i,j} = \sum_{k,\ell=-K}^{K}
   \hat{f}_{k,\ell} \exp \left( 2 \pi {\rm i} ( k x_i + \ell y_j ) \right) ,
   \label{spectral-coll-1}
\end{equation}
its collocation Fourier spectral approximations to first and second order partial derivatives in the $x$-direction become
\begin{eqnarray}
  \left( {\cal D}_{Nx} f \right)_{i,j} = \sum_{k,\ell=-K}^{K}
   \left( 2 k \pi {\rm i} \right) \hat{f}_{k,\ell}
   \exp \left( 2 \pi {\rm i} ( k x_i + \ell y_j ) \right) ,
   \label{spectral-coll-2-1}
\\
  \left( {\cal D}_{Nx}^2 f \right)_{i,j} = \sum_{k,\ell=-K}^{K}
   \left( - 4 \pi^2 k^2 \right) \hat{f}_{k,\ell}
   \exp \left( 2 \pi {\rm i} ( k x_i + \ell y_j) \right) .
   \label{spectral-coll-2-3}
\end{eqnarray}
The differentiation operators in the $y$ direction, namely, ${\cal D}_{Ny}$ and ${\cal D}_{Ny}^2$, could be defined in the same fashion. In turn, the discrete Laplacian, gradient
and divergence become
\begin{eqnarray}
  \Delta_N f =  \left( {\cal D}_{Nx}^2  + {\cal D}_{Ny}^2 \right) f ,  \nonumber
\\
  \nabla_N f = \left(  \begin{array}{c}
  {\cal D}_{Nx} f  \\
  {\cal D}_{Ny} f
  \end{array}  \right)  ,  \quad
  \nabla_N \cdot \left(  \begin{array}{c}
  f _1 \\
  f _2
  \end{array}  \right)  = {\cal D}_{Nx} f_1 + {\cal D}_{Ny} f_2 ,
  \label{spectral-coll-3}
\end{eqnarray}
at the point-wise level. 
See the derivations in the related references~\cite{Boyd2001, canuto82, Gottlieb1977, HGG2007}, etc.

  Given any periodic grid functions $f$ and $g$ (over the 2-D numerical grid), the spectral approximations to the $L^2$ inner product and $L^2$ norm are introduced as
\begin{eqnarray}
  \left\| f \right\|_2 = \sqrt{ \left\langle f , f \right\rangle } ,  \quad \mbox{with} \quad
  \left\langle f , g \right\rangle  = h^2 \sum_{i,j=0}^{N -1}   f_{i,j} g_{i,j} .
  \label{spectral-coll-inner product-1}
\end{eqnarray}
A careful calculation yields the following formulas of summation by parts at the discrete level (see the related discussions~\cite{chen12, chen14, gottlieb12a, gottlieb12b}):
\begin{eqnarray}
  \left\langle f ,  \Delta_N  g  \right\rangle
  = - \left\langle \nabla_N f ,  \nabla_N g   \right\rangle  ,    \quad
  \left\langle f ,  \Delta_N^2  g  \right\rangle
  =  \left\langle \Delta_N f ,  \Delta_N g   \right\rangle  .
  \label{spectral-coll-inner product-3}
\end{eqnarray}
Similarly, for any grid function $f$ with $\overline{f} := h^2 \sum_{i,j=0}^{N-1} f_{i,j} = 0$, the operator $(-\Delta_N)^{-1}$ and the discrete $\| \cdot \|_{-1}$ norm are defined as
\begin{eqnarray}
  &&
  \left( (-\Delta _N)^{-1} f \right)_{i,j} = \sum_{k,\ell \ne \0}
   \frac{1}{\lambda_{k,\ell}} \hat{f}_{k,\ell}
   \exp \left( 2 \pi {\rm i} ( k x_i + \ell y_j) \right) ,  \quad
   \lambda_{k,\ell} = (2 k \pi)^2 + (2 \ell \pi)^2 ,
   \label{spectral-coll-6-1}
 \\
  &&
  \left\| f \right\|_{-1} = \sqrt{ \left\langle f , (-\Delta_N )^{-1} f \right\rangle }   .
   \label{spectral-coll-inner product-4}
\end{eqnarray}



In addition to the standard $\ell^2$ norm, we also introduce the $\ell^p$ and discrete maximum norms for a grid function $f$, to facilitate the analysis in later sections:
\begin{equation}
 \nrm{f}_{\infty} := \max_{i,j} |f_{i,j}| ,   \qquad
 \nrm{f}_{p}  := \Bigl( h^2 \sum_{i,j=0}^{N-1} |f_{i,j} |^p \Bigr)^{\frac{1}{p}} , \quad 1\leq p < \infty.  \label{spectral-defi-Lp}
\end{equation}
Moreover, for any numerical solution $\phi$, the discrete energy is defined as
	\be
E_N (\phi) = E_{c,1,N} (\phi) + \frac{\varepsilon^2}{2} \nrm{\Delta_N \phi}_2^2 \  , \quad E_{c,1,N} (\phi) = h^2 \sum_{i,j=0}^{N-1} \left( - \frac12 \ln \left( 1 + \left| \nabla_N \phi \right|^2 \right)_{i,j} \right) .
	\label{energy-discrete-spectral}
	\ee

\subsection{The proposed third order BDF numerical scheme}

  As usual, we denote $u^k$ as the numerical approximation to the PDE solution at time step $t^k := k \dt$, with any integer $n$. Given $u^n$, $u^{n-1}$, $u^{n-2}$, we propose a third order BDF-type scheme for the NSS equation~\eqref{equation-NSS}:
\begin{eqnarray}
  &&
  \frac{\frac{11}{6} u^{n+1} - 3 u^n + \frac32 u^{n-1} - \frac13 u^{n-2}}{\dt}
  +\varepsilon^2 \Delta_N^2 u^{n+1} +
 \nabla_N \cdot \Bigl( 3 \frac{\nabla_N u^n}{1+ |\nabla_N u^n|^2}  \nonumber
\\
  &&  \quad
  -3 \frac{\nabla_N u^{n-1}}{1+ |\nabla_N u^{n-1}|^2}
  + \frac{\nabla_N u^{n-2}}{1+ |\nabla_N u^{n-2}|^2} \Bigr)
  +A \dt^2 \Delta_N^2( u^{n+1} - u^n) = 0 .
  \label{scheme-BDF3-0}
\end{eqnarray}

\begin{rem}
Since the third order algorithm~\eqref{scheme-BDF3-0} is a three-step scheme, the ``ghost" point approximations $u^{-1}$ and $u^{-2}$ are needed in the initial time step. If we take $u^{-2} = u^{-1} = u^0$, the energy bound estimate~\eqref{energy bound-1} becomes very simple, while the third order numerical accuracy may have been lost in the initial step. Instead, we could use alternate explicit high-order numerical algorithms, such as $RK2$ and $RK3$, to update the numerical solutions at $u^1$ and $u^2$, so that the third order numerical accuracy is preserved in the first few time steps. This approach enables one to derive the full third order temporal convergence estimate in the above theorem. In particular, we notice that the first two $RK$ time steps in the initial approximation will not cause any stability concern, since they are treated as the initial values in the numerical scheme.
\end{rem}

\begin{prop}  \label{prop:mass conservation}
For any initial data with $\overline{u^0} = \overline{u^1} = \overline{u^2} = \beta_0$, the third order numerical scheme~\eqref{scheme-BDF3-0} is mass conservative, i.e, $\overline{u^{n+1}} = \overline{u^n} = ... = \beta_0$ for any $n \ge 2$.
\end{prop}

\begin{proof}
First of all, the following identity is valid, due to the periodic boundary condition for a vector $\f$ and a scalar function $g$:
\begin{eqnarray}
  \overline{ \nabla_N \cdot \f } = 0 ,  \quad
  \overline{\Delta_N g} = \overline{\nabla_N \cdot \nabla_N g} =0  .
\end{eqnarray}
In turn, for the third order scheme~\eqref{scheme-BDF3-0}, all the updated terms have a zero-mean:
\begin{equation}
\begin{aligned}
  &  \Delta_N^2 u^{n+1} = 0 ,  \quad   \Delta_N^2( u^{n+1} - u^n) = 0 ,
\\
  & \nabla_N \cdot \Bigl( 3 \frac{\nabla_N u^n}{1+ |\nabla_N u^n|^2}
    -3 \frac{\nabla_N u^{n-1}}{1+ |\nabla_N u^{n-1}|^2}
  + \frac{\nabla_N u^{n-2}}{1+ |\nabla_N u^{n-2}|^2} \Bigr)  = 0 .
\end{aligned}
\end{equation}
As a consequence, a substitution into~\eqref{scheme-BDF3-0} implies that
\begin{equation}
  \frac{11}{6} \overline{u^{n+1}} - 3 \overline{u^n}
  + \frac32 \overline{u^{n-1}} - \frac13 \overline{u^{n-2}} = 0 ,  \quad \forall n \ge 2.
\end{equation}
With an application of induction analysis, we conclude that
\begin{equation}
    \overline{u^{n+1}} = \overline{u^n} = \overline{u^{n-1}} = \overline{u^{n-2}} = ...
    = \overline{u^0} = \beta_0 ,
\end{equation}
provided that $\overline{u^0} = \overline{u^1} = \overline{u^2} = \beta_0$. This finishes the proof of Proposition~\ref{prop:mass conservation}.
\end{proof}

\subsection{The energy stability analysis}


The following two preliminary estimates in~\cite{Ju18} will be useful in the energy stability analysis.

\begin{lem} \cite{Ju18} \label{lem:vector inequality}
Denote a mapping ${\bf \beta}: R^2 \to R^2$: ${\bf \beta} (\v) = \frac{\v}{1 + | \v|^2 }$. Then we have
\begin{eqnarray}
  | {\bf \beta} (\v) - {\bf \beta} (\w) | \le | \v - \w | ,  \quad \forall \v, \, \w \, \in R^2 .
  \label{vector inequality-0}
\end{eqnarray}
\end{lem}

\begin{lem} \cite{Ju18} \label{lem:convexity}
Define $H (a,b) = \frac12 \ln (1 + a^2 + b^2) + \frac{\kappa_0}{2} (a^2 + b^2)$. Then $H (a,b)$ is convex in $R^2$ if and only if $\kappa_0 \ge \frac18$.
\end{lem}

As a result of these convexity results, we are able to obtain the following energy estimate.

\begin{lem} \label{lem: energy stab-0}
For the numerical solutions $u^{n+1}$ and $u^n$ with $\overline{u^{n+1}} = \overline{u^n}$, we have
\begin{eqnarray}
  &&
   \Bigl\langle \nabla_N \cdot \Big( \frac{\nabla_N u^n}{1+|\nabla_N u^n|^2} \Big)  ,
   u^{n+1} - u^n \Bigr\rangle  \nonumber
\\
  &\ge&
    E_{c,1,N} (u^{n+1}) - E_{c,1,N} (u^n)
   - \frac{\kappa_0}{2} \| \nabla_N ( u^{n+1} - u^n ) \|_2^2 ,
   \label{energy stab-0}
\end{eqnarray}
with the nonlinear energy functional $E_{c,1,N} (\phi)$ defined in~\eqref{energy-discrete-spectral}.
\end{lem}

\begin{proof}
For simplicity of presentation, we denote
\begin{eqnarray}
  f_N^{(0)} (u) = \nabla_N \cdot \left( \frac{\nabla_N u}{1+|\nabla_N u|^2}\right) + \kappa _0 \Delta_N u . \label{energy stab-1}  
\end{eqnarray}
By Lemma~\ref{lem:convexity}, $f_N^{(0)} (u)$ corresponds to a concave energy functional, so that the following convexity inequality is valid:
\begin{eqnarray}
  \langle f_N^{(0)} (u^n) ,  u^{n+1} - u^n \rangle \ge H_{N} (u^{n+1}) - H_N (u^n) , \quad \mbox{with} \, \, \, H_N (\phi) =  E_{c,1,N} (\phi) - \frac{\kappa_0}{2} \| \nabla_N \phi \|_2^2 ,  \label{energy stab-2}
\end{eqnarray}
which is equivalent to~\eqref{energy stab-0}. This finishes the proof of Lemma~\ref{lem: energy stab-0}.
\end{proof}

The energy stability of the proposed third order BDF-type scheme~\eqref{scheme-BDF3-0}  is stated in the following theorem, in a modified version.  

\begin{thm}  \label{thm:energy stab-BDF3}
  The numerical solution produced by the proposed BDF-type scheme~\eqref{scheme-BDF3-0} satisfies
\begin{eqnarray}
  \tilde{E}_N (u^{n+1} , u^n, u^{n-1}) &\le& \tilde{E}_N (u^n, u^{n-1}, u^{n-2}) ,   \quad \mbox{with}  \nonumber
\\
  \tilde{E}_N (u^{n+1} , u^n, u^{n-1}) &=& E_N (u^{n+1})
   + \frac{3}{4 \dt} \| u^{n+1} - u^n \|_2^2 + \frac{1}{6 \dt} \| u^n - u^{n-1} \|_2^2   \nonumber
\\
  &&
  + \frac32 \| \nabla_N ( u^{n+1} - u^n ) \|_2^2
  + \frac12 \| \nabla_N ( u^n - u^{n-1} ) \|_2^2  ,  \label{scheme-BDF3-stability-0}
\end{eqnarray}
for any $\dt >0$, provided that $A \ge \frac{9}{32} ( \frac{49}{16} )^4 \varepsilon^{-2}$.
\end{thm}

\begin{proof}
The numerical scheme~\eqref{scheme-BDF3-0} could be rewritten as
\begin{equation} \label{scheme-BDF3-stability-1}
\begin{split}
\nabla_N\cdot \Big( \frac{\nabla_N u^n}{1+|\nabla_N u^n|^2} \Big) =&-\frac{\frac{11}{6} u^{n+1} - 3 u^n +\frac32 u^{n-1} - \frac13 u^{n-2}}{\dt}
\\
  &
  - \nabla_N \cdot \Bigl( 2 \frac{\nabla_N u^n}{1+|\nabla_N u^n|^2}
  - 3 \frac{\nabla_N u^{n-1}}{1+|\nabla_N u^{n-1}|^2} + \frac{\nabla_N u^{n-2}}{1+|\nabla_N u^{n-2}|^2} \Bigr)
\\
  &
  - \varepsilon^2 \Delta_N^2 u^{n+1}  - A \dt^2 \Delta_N^2(u^{n+1} - u^n) .
\end{split}
\end{equation}
Taking a discrete $\ell^2$ inner product with~\eqref{scheme-BDF3-stability-1} by $u^{n+1}-u^n$ yields
\begin{equation} \label{scheme-BDF3-stability-2}
\begin{split}
&\Bigl\langle \nabla_N\cdot \Big( \frac{\nabla_N u^n}{1+|\nabla_N u^n|^2} \Big) ,
  u^{n+1} - u^n \Bigr\rangle
+ \frac{1}{\dt} \Big\langle \frac{11}{6} u^{n+1} - 3 u^n + \frac32 u^{n-1}
-\frac13 u^{n-2} , u^{n+1} - u^n \Big\rangle \\
+&
\Bigl\langle \nabla_N \cdot \Big( 2 \frac{\nabla_N u^n}{1+|\nabla_N u^n|^2}
- 3 \frac{\nabla_N u^{n-1}}{1+|\nabla_N u^{n-1}|^2}
+ \frac{\nabla_N u^{n-2}}{1+|\nabla_N u^{n-2}|^2} \Big), u^{n+1} - u^n \Bigr\rangle \\
+&
 \varepsilon^2 \langle \Delta_N^2 u^{n+1}, u^{n+1} - u^n\rangle
 + A \dt^2\langle \Delta_N^2 (u^{n+1} - u^n), u^{n+1} - u^n \rangle = 0 .
\end{split}
\end{equation}
The temporal stencil term could be analyzed as follows:
\begin{equation}  \label{scheme-BDF3-stability-3}
\begin{split}
& \Big\langle\frac{11}{6} u^{n+1} - 3 u^n + \frac32 u^{n-1}
- \frac13 u^{n-2}, u^{n+1} - u^n \Big\rangle \\
=& \Big\langle \frac{11}{6}( u^{n+1} - u^n) - \frac76 ( u^n - u^{n-1})
+ \frac13 (u^{n-1} - u^{n-2}), u^{n+1} - u^n \Big\rangle
\\
\ge& \frac{11}{6} \| u^{n+1} - u^n\|^2_2
-\frac{7}{12}( \| u^n - u^{n-1}\|^2_2
+ \| u^{n+1} - u^n\|^2_2 )
\\&-\frac16 (\| u^{n-1} - u^{n-2}\|^2_2 + \| u^{n+1} - u^n\|^2_2 )
\\
=& \frac13 \| u^{n+1} - u^n\|^2_2
+\frac34 (\| u^{n+1} - u^n\|^2_2 - \| u^n - u^{n-1}\|^2_2)
\\
 &  + \frac16 (\| u^n - u^{n-1}\|^2_2 - \| u^{n-1} - u^{n-2}\|^2_2 ) .
\end{split}
\end{equation}
For the nonlinear increment term, the following estimate could be derived:
\begin{equation}  \label{scheme-BDF3-stability-4}
\begin{split}
& - \Bigl\langle 2 \frac{\nabla_N u^n}{1+|\nabla_N u^n|^2}
 - 3 \frac{\nabla_N u^{n-1}}{1+|\nabla_N u^{n-1}|^2}
 + \frac{\nabla_N u^{n-2}}{1+|\nabla_N u^{n-2}|^2},
 \nabla_N( u^{n+1} - u^n) \Bigr\rangle
\\=& - 2 \Bigl\langle \frac{\nabla_N u^n}{1+|\nabla_N u^n|^2} - \frac{\nabla_N u^{n-1}}{1+|\nabla_N u^{n-1}|^2}, \nabla_N (u^{n+1} - u^n) \Bigr\rangle
\\& + \Bigl\langle \frac{\nabla_N u^{n-1}}{1+|\nabla_N u^{n-1}|^2}
 - \frac{\nabla_N u^{n-2}}{1+|\nabla_N u^{n-2}|^2},
  \nabla_N( u^{n+1} - u^n) \Big\rangle
\\
  \le&
  2 \|\nabla_N ( u^n - u^{n-1})\|_2 \cdot \|\nabla_N ( u^{n+1} - u^n) \|_2
  + \|\nabla_N ( u^{n-1} - u^{n-2} )\|_2 \cdot \|\nabla_N ( u^{n+1} - u^n) \|_2
\\
  \le&
  \frac32 \|\nabla_N ( u^{n+1} - u^n) \|^2_2
+ \|\nabla_N ( u^n - u^{n-1} ) \|^2_2
+\frac12 \| \nabla_N ( u^{n-1} - u^{n-2} )\|^2_2 ,
\end{split}
\end{equation}
in which the point-wise inequality~\eqref{vector inequality-0} (in Lemma~\ref{lem:vector inequality}) has been applied in the second step. For the surface diffusion term, the following identity is available:
\begin{equation}   \label{scheme-BDF3-stability-5}
 \langle \Delta_N^2 u^{n+1}, u^{n+1} - u^n \rangle
= \frac12 ( \|\Delta_N u^{n+1} \|^2_2 - \|\Delta_N u^n \|_2^2
 + \| \Delta_N ( u^{n+1} - u^n ) \|^2_2 ) .
\end{equation}
Similarly, the following identity is straightforward to the artificial Douglas-Dupont regularization term:
\begin{equation}  \label{scheme-BDF3-stability-6}
 A \dt^2 \langle \Delta_N^2 (u^{n+1} - u ^n), u^{n+1} - u^n \rangle
= A \dt^2 \| \Delta_N ( u^{n+1} - u^n ) \|^2_2 .
\end{equation}
As a consequence, a substitution of~\eqref{energy stab-0} and \eqref{scheme-BDF3-stability-2}-\eqref{scheme-BDF3-stability-6} into \eqref{scheme-BDF3-stability-1} results in
\begin{equation} \label{scheme-BDF3-stability-7}
\begin{split}
& E_N (u^{n+1}) - E_N (u^n)
  + ( \frac{\varepsilon^2}{2} + A \dt^2 ) \| \Delta_N ( u^{n+1} - u^n )\|^2_2
   + \frac{1}{3 \dt} \| u^{n+1} - u^n\|^2_2
\\
 & +\frac{3}{4 \dt} (\| u^{n+1} - u^n\|^2_2 - \| u^n - u^{n-1}\|^2_2)
+ \frac{1}{6 \dt} (\| u^n - u^{n-1}\|^2_2 - \| u^{n-1} - u^{n-2}\|^2_2 )
\\
  \le&
  \frac{25}{16} \|\nabla_N ( u^{n+1} - u^n) \|^2_2
+ \|\nabla_N ( u^n - u^{n-1} ) \|^2_2
+\frac12 \| \nabla_N ( u^{n-1} - u^{n-2} )\|^2_2 ,
\end{split}
\end{equation}
with the optimal value of $\kappa_0 = \frac18$ taken. To control the right hand side of~\eqref{scheme-BDF3-stability-7}, we begin with the following quadratic inequality
\begin{equation}  \label{scheme-BDF3-stability-8-1}
    \frac{\varepsilon^2}{2} + A \dt^2
    \ge \sqrt{2} A^{1/2} \varepsilon \dt ,
\end{equation}
which in turn leads to
\begin{equation} \label{scheme-BDF3-stability-8-2}
\begin{split}
& ( \frac{\varepsilon^2}{2} + A \dt^2 ) \| \Delta_N ( u^{n+1} - u^n )\|^2_2
   + \frac{1}{3 \dt} \| u^{n+1} - u^n\|^2_2
\\
 \ge &  \sqrt{2} A^{1/2} \varepsilon \dt  \| \Delta_N ( u^{n+1} - u^n )\|^2_2
   + \frac{1}{3 \dt} \| u^{n+1} - u^n\|^2_2
\\
  \ge&
  2^{5/4} 3^{-1/2} A^{1/4} \varepsilon^{1/2}   \| \Delta_N ( u^{n+1} - u^n )\|_2
   \cdot \| u^{n+1} - u^n\|_2
\\
  \ge &
   2^{5/4} 3^{-1/2} A^{1/4} \varepsilon^{1/2}   \| \nabla_N ( u^{n+1} - u^n ) \|_2^2 ,
\end{split}
\end{equation}
in which the summation by parts formula has been applied in the last step, as well as the following estimate:
\begin{equation}
  \| \nabla_N ( u^{n+1} - u^n ) \|_2^2
  = - \langle  u^{n+1} - u^n , \Delta_N ( u^{n+1} - u^n ) \rangle
  \le   \| \Delta_N ( u^{n+1} - u^n )\|_2 \cdot \| u^{n+1} - u^n\|_2  .
\end{equation}
Under the constraint that
\begin{equation} \label{constraint-1}
   2^{5/4} 3^{-1/2} A^{1/4} \varepsilon^{1/2}   \ge \frac{49}{16} ,  \quad \mbox{i.e.} \, \, \,
   A \ge \frac{9}{32} ( \frac{49}{16} )^4 \varepsilon^{-2} ,
\end{equation}
the following inequality is valid:
\begin{equation} \label{scheme-BDF3-stability-9}
\begin{split}
& E_N (u^{n+1}) - E_N (u^n)
  +  \frac{49}{16} \| \nabla_N ( u^{n+1} - u^n ) \|_2^2
\\
 & +\frac{3}{4 \dt} (\| u^{n+1} - u^n\|^2_2 - \| u^n - u^{n-1}\|^2_2)
+ \frac{1}{6 \dt} (\| u^n - u^{n-1}\|^2_2 - \| u^{n-1} - u^{n-2}\|^2_2 )
\\
  \le&
  \frac{25}{16} \|\nabla_N ( u^{n+1} - u^n) \|^2_2
+ \|\nabla_N ( u^n - u^{n-1} ) \|^2_2
+\frac12 \| \nabla_N ( u^{n-1} - u^{n-2} )\|^2_2 .
\end{split}
\end{equation}
In fact, \eqref{scheme-BDF3-stability-9} is equivalent to
\begin{equation}  \label{scheme-BDF3-stability-10}
  \tilde{E}_N (u^{n+1} , u^n, u^{n-1}) - \tilde{E}_N (u^n, u^{n-1}, u^{n-2}) \le 0 .
\end{equation}
This finishes the proof of Theorem~\ref{thm:energy stab-BDF3}.
\end{proof}

\begin{cor}  \label{cor:energy bound-BDF3}
  For the numerical solution~\eqref{scheme-BDF3-0}, we have
\begin{eqnarray}
  E_N (u^k) &\le& E_N (u^0) + \frac{3}{4 \dt} \| u^0 - u^{-1} \|_2^2
  + \frac{1}{6 \dt} \| u^{-1} - u^{-2} \|_2^2  \nonumber
\\
  &&
  + \frac32 \| \nabla_N (u^0 - u^{-1} ) \|_2^2
  + \frac12 \| \nabla_N (u^{-1} - u^{-2}) \|_2^2 := \tilde{C}_0 , \quad \forall k \ge 0 ,
  \label{energy bound-0}
\end{eqnarray}
provided that~\eqref{constraint-1} is satisfied.
\end{cor}

\begin{proof}
By the modified energy inequality~\eqref{scheme-BDF3-stability-0}, the following induction analysis could be performed:
\begin{eqnarray}
  \hspace{-0.35in} &&
  E_N (u^k ) \le \tilde{E}_N (u^k , u^{k-1}, u^{k-2}) \le ... \le  \tilde{E}_N (u^0 , u^{-1}, u^{-2}) := \tilde{C}_0 ,  \quad \forall k \ge 0 .
  \label{energy bound-1}
\end{eqnarray}
\end{proof}

\begin{rem}
By the detailed expansion~\eqref{scheme-BDF3-stability-0} of the modified energy $\tilde{E}_N (u^{n+1}, u^n, u^{n-1})$, it is observed to be an $O (\dt)$ approximation to the original energy functional $E_N (u^{n+1})$ (at time step $t^{n+1}$), not a third order accurate approximation. In fact, the $O (\dt)$ approximation order comes from two correction terms, namely, $\frac{3}{4 \dt} \| u^{n+1} - u^n \|_2^2$ and $\frac{1}{6 \dt} \| u^n - u^{n-1} \|_2^2$. An additional stability property in the temporal stencil enables us to derive such a modified energy stability. On the other hand, although there is $O (\dt)$ difference between the modified and original energy functionals, such a modified stability ensures a uniform-in-time bound for the original energy functional, due to the non-negative feature of all the correction terms.
\end{rem}

\begin{rem}
The requirement~\eqref{constraint-1} for the parameter $A$ indicates an order of $A=O (\varepsilon^{-2})$. Such a requirement is based on a subtle fact that, an extra stability estimate from the surface diffusion term has to be used to balance the stability loss coming from the multi-step explicit treatment of the nonlinear terms, and the surface diffusion coefficient is given by $\varepsilon^2$ in the physical parameter.

On the other hand, such a parameter order $A= O (\varepsilon^{-2})$ is only used for the theoretical justification of the energy stability. In the practical computations, an extensive choice of $A= O (1)$ has never led to any energy stability loss for the proposed third order scheme, although its stability could not be theoretically justified. In fact, the effect of the artificial regularization coefficients has been investigated in a recent article~\cite{Meng2020}, for the second order energy stable BDF2 scheme for the NSS equation~\eqref{equation-NSS}. The theoretical analysis indicates a minimum value of $A_0 = \frac{289}{1024}$ to obtain a modified energy stability for the second order BDF-type scheme, while the numerical simulation with an even smaller $A_0 = 0.1$ has yielded more accurate results in the long time computations. Similar behaviors have also been observed in the proposed third order BDF-type scheme~\eqref{scheme-BDF3-0}: The theoretical analysis indicates a minimum value of $A_0 = \frac{9}{32} (\frac{49}{16})^4 \varepsilon^{-2}$ to obtain a modified energy stability for the proposed third order scheme, while the numerical simulation with $A_0 = O (1)$ has yielded more accurate results in the long time computations. In the numerical results presented in Section~\ref{sec:numerical results}, an $O (1)$ artificial regularization coefficient will be utilized.
\end{rem}

\begin{rem}
There have been a few recent works of the second order BDF schemes for certain gradient flow models, such as Allen-Cahn and Cahn-Hilliard~\cite{cheng2019a, Liao2020b, yan18}, slope-selection thin film equation~\cite{fengW18b}, square phase field crystal~\cite{cheng2019d}, in which the energy stability was theoretically established. Similarly, a Douglas-Dupont type regularization has to be included in the numerical scheme, while a careful analysis only requires the corresponding parameter $A$ of an order $A=O (1)$.  The primary reason for the difference in the order of the artificial parameter $A$ between the second and third order numerical schemes is based on the following fact: for the second order scheme, the artificial regularization, with magnitude $O(\dt^2)$, and the temporal discretization terms are sufficient to theoretically justify the energy stability; while for the third order scheme, these two terms are not sufficient to ensure the numerical stability, since the artificial regularization term has to be in the order of $O(\dt^3)$ to keep the third order temporal accuracy. 
\end{rem}



\begin{rem}
The stability and convergence estimates for the temporally third order accurate numerical schemes have been reported for fluid models, such as viscous Burgers' equation~\cite{gottlieb12b}, incompressible Navier-Stokes equation~\cite{cheng16b}, harmonic mapping flow~\cite{Xia2020a}, etc.

For the gradient models, the only existing works to address the energy stability for a third order numerical scheme could be found in~\cite{chen20b, cheng2019c, Gong2020, xu06}. In this article, we provide an alternate third order numerical approach for the NSS equation, for which both the energy stability and optimal rate convergence estimate could be theoretically justified.
\end{rem}

\begin{rem}
The second Dahlquist barrier~\cite{Dahlquist1963} states that, ``There are no explicit A-stable and linear multistep methods.  The implicit ones have order of convergence at most 2." The proposed third order accurate scheme~\eqref{scheme-BDF3-0} is based on the BDF3 temporal approximation for the surface diffusion part, while it is well-known that the BDF3 method is not A-stable. On the other hand, although the BDF3 method is not A-stable, an A-stability is not necessary to preserve an unconditional energy stability for the NSS equation~\eqref{equation-NSS}. In more details, the domain of the BDF3 contains a very large portion of the left half complex plane. In particular, the whole negative real axis (where all the eigenvalues associated with the surface diffusion operator are located) is contained inside the stability domain. Furthermore, the nonlinear chemical potential parts have automatically bounded higher order derivatives, and this subtle fact leads to an unconditional energy stability, with the help of extra dissipations coming from the surface diffusion and temporal differentiation terms. More importantly, an addition of artificial regularization in the numerical scheme~\eqref{scheme-BDF3-0} makes it not equivalent to the original BDF3 algorithm, and this term results in better stability property than the standard BDF3 method. All these facts yield the desired unconditional energy stability of the proposed third order scheme~\eqref{scheme-BDF3-0}, and no CFL-like condition is needed for the time step size.
\end{rem}

\section{The convergence analysis for the third order BDF scheme}
\label{sec-convergence}

The global existence of weak solution, strong solution and smooth solution for the NSS equation (\ref{equation-NSS}) has been established in~\cite{libo03}. In more details, a global-in-time estimate of $L^\infty (0,T; H^m) \cap L^2 (0,T; H^{m+2})$ for the phase variable was proved, assuming initial data in $H^m$, for any $m \ge 2$. Therefore, with an initial data with sufficient regularity, we could assume that the exact solution has regularity of class $\mathcal{R}$:
		\begin{equation}
			u_{e} \in \mathcal{R} := H^4 (0,T; C^0) \cap H^1 (0,T; H^4) \cap H^3 (0,T; H^{m+2}) \cap L^\infty (0,T; H^{m+4}).
			\label{assumption:regularity.1}
		\end{equation}


Define $U_N (\, \cdot \, ,t) := {\cal P}_N u_e (\, \cdot \, ,t)$, the (spatial) Fourier projection of the exact solution into ${\cal B}^K$, the space of trigonometric polynomials of degree to and including  $K$.  The following projection approximation is standard: if $u_e \in L^\infty(0,T;H^\ell_{\rm per}(\Omega))$, for some $\ell\in\mathbb{N}$, we have
	\begin{equation}
\| U_N - u_e \|_{L^\infty(0,T;H^k)}
   \le C h^{\ell-k} \| u_e \|_{L^\infty(0,T;H^\ell)},  \quad \forall \ 0 \le k \le \ell .
	\label{projection-est-0}
	\end{equation}
By $U_N^m$ we denote $U_N(\, \cdot \, , t^m)$, with $t^m = m\cdot \dt$. Since $U_N \in {\cal P}_K$, the mass conservative property is available at the discrete level:
	\begin{equation}
\overline{U_N^m} = \frac{1}{|\Omega|}\int_\Omega \, U_N ( \cdot, t_m) \, d {\bf x} = \frac{1}{|\Omega|}\int_\Omega \, U_N ( \cdot, t_{m-1}) \, d {\bf x} = \overline{U_N^{m-1}} ,  \quad \forall \ m \in\mathbb{N}.
	\label{mass conserv-1}
	\end{equation}
On the other hand, the solution of the numerical scheme~\eqref{scheme-BDF3-0} is also mass conservative at the discrete level:
	\begin{equation}
\overline{u^m} = \overline{u^{m-1}} ,  \quad \forall \ m \in \mathbb{N} .
	\label{mass conserv-2}
	\end{equation}
Meanwhile, we denote $U^m$ as the interpolation values of $U_N$ at discrete grid points at time instant $t^m$: $U_{i,j}^m :=  U_N (x_i, y_j, t^m)$.
As indicated before, we use the mass conservative projection for the initial data:  
	\begin{equation}
u^0_{i,j} = U^0_{i,j} := U_N (x_i, y_j, t=0) .
	\label{initial data-0}
	\end{equation}	
The error grid function is defined as
	\begin{equation}
e^m := U^m - u^m ,  \quad \forall \ m \in \left\{ 0 ,1 , 2, 3, \cdots \right\} .
	\label{error function-1}
	\end{equation}
Therefore, it follows that  $\overline{e^m} =0$, for any $m \in \left\{ 0 ,1 , 2, 3, \cdots \right\}$. 

For the proposed third order BDF-type scheme~\eqref{scheme-BDF3-0}, the convergence result is stated below.

    \begin{thm}
    	\label{thm:convergence}
    	Given initial data $U_N^{0}$, $U_N^{-1}$, $U_N^{-2} \in C_{\rm per}^{m+4} (\overline{\Omega})$, with periodic boundary conditions, suppose the unique solution for the NSS equation (\ref{equation-NSS}) is of regularity class $\mathcal{R}$. Then, provided $\dt$ and $h$ are sufficiently small, 
for all positive integers $\ell$, such that $\dt \cdot \ell \le T$, we have
    	\begin{eqnarray}
       &&	
    	\| e^\ell \|_2 +  \Bigl( \varepsilon^2 \dt   \sum_{m=1}^{\ell} \| \Delta_N e^m \|_2^2 \Bigr)^{1/2} \le C ( \dt^3 + h^m ),   \label{convergence-0-2}
    	\end{eqnarray}
    	where $C>0$ is independent of $\dt$ and $h$.
    \end{thm}

\subsection{The error evolutionary equation}

  For the Fourier projection solution $U_N$ and its interpolation $U$, a careful consistency analysis implies that
\begin{eqnarray}
  &&
  \frac{\frac{11}{6} U^{n+1} - 3 U^n + \frac32 U^{n-1} - \frac13 U^{n-2}}{\dt}
  +\varepsilon^2 \Delta_N^2 U^{n+1} +A \dt^2 \Delta_N^2( U^{n+1} - U^n) \nonumber
\\
  &&
 = - \nabla_N \cdot \Bigl( 3 \frac{\nabla_N U^n}{1+ |\nabla_N U^n|^2}
  -3 \frac{\nabla_N U^{n-1}}{1+ |\nabla_N U^{n-1}|^2}
  + \frac{\nabla_N U^{n-2}}{1+ |\nabla_N U^{n-2}|^2} \Bigr) + \tau^n ,
  \label{BDF3-consistency-1}
\end{eqnarray}
with $\| \tau^n \|_2 \le C (\dt^3 + h^m)$. In turn, subtracting the numerical scheme \eqref{scheme-BDF3-0} from the consistency estimate \eqref{BDF3-consistency-1} yields
\begin{eqnarray}
  &&
  \frac{\frac{11}{6} e^{n+1} - 3 e^n + \frac32 e^{n-1} - \frac13 e^{n-2}}{\dt}
  +\varepsilon^2 \Delta_N^2 e^{n+1} +A \dt^2 \Delta_N^2( e^{n+1} - e^n) \nonumber
\\
  &=&
   - \nabla_N \cdot \Bigl( 3 \frac{\nabla_N U^n}{1+ |\nabla_N U^n|^2}
  - 3 \frac{\nabla_N u^n}{1+ |\nabla_N u^n|^2}
  -3 \frac{\nabla_N U^{n-1}}{1+ |\nabla_N U^{n-1}|^2}
  +3 \frac{\nabla_N u^{n-1}}{1+ |\nabla_N u^{n-1}|^2}    \nonumber
\\
  &&
  + \frac{\nabla_N U^{n-2}}{1+ |\nabla_N U^{n-2}|^2}
  - \frac{\nabla_N u^{n-2}}{1+ |\nabla_N u^{n-2}|^2}  \Bigr) + \tau^n .
  \label{BDF3-consistency-2}
\end{eqnarray}

\subsection{The $\ell^\infty (0,T; \ell^2) \cap \ell^2 (0,T; H_h^2)$ error estimate}

Before the proof of the convergence result, we present the telescope formula in \cite{JLiu2013} for the third order BDF temporal discretization operator in the following lemma; also see~\cite{yao17} for the related discussion. 
\begin{lem} \label{lem:3rd order BDF}
For the third order BDF temporal discretization operator, there exists $\alpha_i$, $i=1,\cdots,10$, $\alpha_1 \ne 0$, such that
\begin{eqnarray} \label{BDF-3-est-0}
&& \langle \frac{11}{6} e^{n+1} - 3 e^n + \frac32 e^{n-1} - \frac13 e^{n-2} ,
 2 e^{n+1} - e^n \rangle \\
&=&  \| \alpha_1 e^{n+1} \|_2^2 - \| \alpha_1 e^n \|_2^2
 + \| \alpha_2 e^{n+1} + \alpha_3 e^n \|_2^2
 - \| \alpha_2 e^n + \alpha_3 e^{n-1} \|_2 ^2 \nonumber
\\
  &&
  + \| \alpha_4 e^{n+1} + \alpha_5 e^n + \alpha_6 e^{n-1} \|_2^2
 - \| \alpha_4 e^n + \alpha_5 e^{n-1} + \alpha_6 e^{n-2} \|_2^2 \nonumber
\\
  &&+ \| \alpha_7 e^{n+1} + \alpha_8 e^n + \alpha_9 e^{n-1}
   + \alpha_{10} e^{n-2} \|_2^2 . \nonumber
\end{eqnarray}
\end{lem}

Now we proceed into the convergence estimate. Taking a discrete $\ell^2$ inner product with~\eqref{BDF3-consistency-2} by $2 e^{n+1} - 2 e^n$ gives
\begin{eqnarray}
  &&
  \frac{1}{\dt} \langle \frac{11}{6} e^{n+1} - 3 e^n + \frac32 e^{n-1} - \frac13 e^{n-2} ,
 2e^{n+1} - e^n \rangle   \nonumber
\\
  &&
  + \varepsilon^2 \langle \Delta_N e^{n+1} , \Delta_N ( 2 e^{n+1} - e^n ) \rangle
  +A \dt^2 \langle \Delta_N ( e^{n+1} - e^n) , \Delta_N ( 2 e^{n+1} - e^n ) \rangle \nonumber
\\
  &=&
   \sum_{j=0}^2 \gamma^{(j)}
   \Big\langle \frac{\nabla_N U^{n-j}}{1+ |\nabla_N U^{n-j}|^2}
  - \frac{\nabla_N u^{n-j}}{1+ |\nabla_N u^{n-j}|^2}  ,
   \nabla_N ( 2 e^{n+1} - e^n ) \Big\rangle
   + \langle \tau^n , 2 e^{n+1} - e^n \rangle  ,
  \label{convergence-1}
\end{eqnarray}
with $\gamma^{(0)} =3$, $\gamma^{(1)} =-3$, $\gamma^{(2)} =1$. Notice that the summation by parts formula has been repeatedly applied in the derivation. The local truncation error term could also be bounded in a straightforward way:
\begin{eqnarray}
    \langle \tau^n , 2 e^{n+1} - e^n \rangle
    \le \| \tau^n \|_2^2 + \| e^{n+1} \|_2^2 + \frac12 ( \| \tau^n \|_2^2 + \| e^n \|_2^2 )
    =  \frac32 \| \tau^n \|_2^2 + \| e^{n+1} \|_2^2 + \frac12 \| e^n \|_2^2 .
   \label{convergence-2}
\end{eqnarray}
The surface diffusion and the Douglas-Dupont regularization terms could be analyzed as follows:
\begin{eqnarray}
    \langle \Delta_N e^{n+1} , \Delta_N ( 2 e^{n+1} - e^n ) \rangle
    &=&  \| \Delta_N e^{n+1} \|_2^2
    + \langle \Delta_N e^{n+1} , \Delta_N ( e^{n+1} - e^n ) \rangle  \nonumber
\\
  &\ge&
  \| \Delta_N e^{n+1} \|_2^2
  + \frac12 ( \| \Delta_N e^{n+1} \|_2^2 - \| \Delta_N e^n \|_2^2  ) ,
  \label{convergence-3}
\\
  \langle \Delta_N ( e^{n+1} - e^n ) , \Delta_N ( 2 e^{n+1} - e^n ) \rangle
    &=&  \| \Delta_N ( e^{n+1} - e^n ) \|_2^2
    + \langle \Delta_N e^{n+1} , \Delta_N ( e^{n+1} - e^n ) \rangle  \nonumber
\\
  &\ge&
  \| \Delta_N ( e^{n+1} - e^n ) \|_2^2
  + \frac12 ( \| \Delta_N e^{n+1} \|_2^2 - \| \Delta_N e^n \|_2^2  ) .
  \label{convergence-4}
\end{eqnarray}

For the nonlinear error term on the right hand side of~\eqref{convergence-1}, we focus on the time instant $t^n$. By making use of inequality~\eqref{vector inequality-0} (in Lemma~\ref{lem:vector inequality}), we see that
\begin{eqnarray}
   \Bigl| \frac{\nabla_N U^n}{1+ |\nabla_N U^n|^2}
  - \frac{\nabla_N u^n}{1+ |\nabla_N u^n|^2}  \Bigr|
  \le | \nabla_N e^n |  ,  \quad \mbox{at a point-wise level} .
  \label{convergence-5-1}
\end{eqnarray}
This in turn leads to
\begin{eqnarray}
   \gamma^{(0)}
   \Big\langle \frac{\nabla_N U^n}{1+ |\nabla_N U^n|^2}
  - \frac{\nabla_N u^n}{1+ |\nabla_N u^n|^2}  ,
   \nabla_N ( 2 e^{n+1} - e^n ) \Big\rangle
   \le 3 \| \nabla_N e^n \|_2 \cdot \|  \nabla_N ( 2 e^{n+1} - e^n ) \|_2 .
   \label{convergence-5-2}
\end{eqnarray}
Similarly, the following inequalities are available:
\begin{eqnarray}
  &&
   \gamma^{(1)}
   \Big\langle \frac{\nabla_N U^{n-1}}{1+ |\nabla_N U^{n-1}|^2}
  - \frac{\nabla_N u^{n-1}}{1+ |\nabla_N u^{n-1}|^2}  ,
   \nabla_N ( 2 e^{n+1} - e^n ) \Big\rangle   \nonumber
\\
  &&  \qquad
   \le 3 \| \nabla_N e^{n-1} \|_2 \cdot \|  \nabla_N ( 2 e^{n+1} - e^n ) \|_2 ,
   \label{convergence-5-3}
\\
  &&
   \gamma^{(2)}
   \Big\langle \frac{\nabla_N U^{n-2}}{1+ |\nabla_N U^{n-2}|^2}
  - \frac{\nabla_N u^{n-2}}{1+ |\nabla_N u^{n-2}|^2}  ,
   \nabla_N ( 2 e^{n+1} - e^n ) \Big\rangle   \nonumber
\\
  &&  \qquad
   \le \| \nabla_N e^{n-2} \|_2 \cdot \|  \nabla_N ( 2 e^{n+1} - e^n ) \|_2 .
   \label{convergence-5-4}
\end{eqnarray}
Then we arrive at
\begin{eqnarray}
   &&
   \sum_{j=0}^2 \gamma^{(j)}
   \Big\langle \frac{\nabla_N U^{n-j}}{1+ |\nabla_N U^{n-j}|^2}
  - \frac{\nabla_N u^{n-j}}{1+ |\nabla_N u^{n-j}|^2}  ,
   \nabla_N ( 2 e^{n+1} - e^n ) \Big\rangle   \nonumber
\\
  &\le&
  ( 3 \| \nabla_N e^n \|_2 + 3 \| \nabla_N e^{n-1} \|_2 + \| \nabla_N e^{n-2} \|_2 )
  \cdot \|  \nabla_N ( 2 e^{n+1} - e^n ) \|_2  \nonumber
\\
  &\le&
   7 \|  \nabla_N e^{n+1} \|_2^2 + 8 \| \nabla_N e^n \|_2^2
   + \frac92 \| \nabla_N e^{n-1} \|_2^2 + \frac32 \| \nabla_N e^{n-2} \|_2^2 ,
   \label{convergence-5-5}
\end{eqnarray}
with repeated application of Cauchy inequality at the last step.

  Therefore, a substitution of~\eqref{BDF-3-est-0}, \eqref{convergence-2}-\eqref{convergence-4} and \eqref{convergence-5-5} into \eqref{convergence-1} leads to
\begin{eqnarray}
  &&
    \frac{1}{\dt} ( F^{n+1} - F^n  ) + \varepsilon^2 \| \Delta_N e^{n+1} \|_2^2
  + \frac12 ( \varepsilon^2 + A \dt^2 )
  ( \| \Delta_N e^{n+1} \|_2^2 - \| \Delta_N e^n \|_2^2  )  \nonumber
\\
  &\le&
   7 \|  \nabla_N e^{n+1} \|_2^2 + 8 \| \nabla_N e^n \|_2^2
   + \frac92 \| \nabla_N e^{n-1} \|_2^2 + \frac32 \| \nabla_N e^{n-2} \|_2^2
    + \frac32 \| \tau^n \|_2^2 + \| e^{n+1} \|_2^2 + \frac12 \| e^n \|_2^2  ,
   \label{convergence-6-1}
\\
  &&  \mbox{with} \, \, \,
      F^{n+1} = \| \alpha_1 e^{n+1} \|_2^2
+ \| \alpha_2 e^{n+1} + \alpha_3 e^n \|_2^2
  + \| \alpha_4 e^{n+1} + \alpha_5 e^n + \alpha_6 e^{n-1} \|_2^2 .  \label{defi-F}
\end{eqnarray}
Meanwhile, for the error gradient term $\|  \nabla_N e^{n+1} \|_2^2$, the following estimate could be derived:
\begin{eqnarray}
  7 \|  \nabla_N e^{n+1} \|_2^2 &=& - 7  \langle e^{n+1},  \Delta_N e^{n+1} \rangle
  \le  7  \| e^{n+1} \|_2 \cdot \|  \Delta_N e^{n+1} \|_2  \nonumber
\\
  &\le&
      \frac{147}{2} \varepsilon^{-2}  \| e^{n+1} \|_2^2
  + \frac{\varepsilon^2}{6} \|  \Delta_N e^{n+1} \|_2^2  .
  \label{convergence-6-2}
\end{eqnarray}
The error gradient terms at three other time step instants could be bounded in a similar way:
\begin{eqnarray}
  8 \|  \nabla_N e^n \|_2^2   &\le&
  96 \varepsilon^{-2}  \| e^n \|_2^2  + \frac{\varepsilon^2}{6} \|  \Delta_N e^n \|_2^2  ,
  \label{convergence-6-3}
\\
  \frac92 \|  \nabla_N e^{n-1} \|_2^2   &\le&
  \frac{243}{8} \varepsilon^{-2}  \| e^{n-1} \|_2^2
  + \frac{\varepsilon^2}{6} \|  \Delta_N e^{n-1} \|_2^2  ,
  \label{convergence-6-4}
\\
  \frac32 \|  \nabla_N e^{n-2} \|_2^2   &\le&
  \frac{27}{8} \varepsilon^{-2}  \| e^{n-2} \|_2^2
  + \frac{\varepsilon^2}{6} \|  \Delta_N e^{n-2} \|_2^2  .
  \label{convergence-6-5}
\end{eqnarray}
Going back~\eqref{convergence-6-1}, we arrive at
\begin{eqnarray}
  &&
    \frac{1}{\dt} ( F^{n+1} - F^n  ) + \frac56 \varepsilon^2 \| \Delta_N e^{n+1} \|_2^2
  + \frac12 ( \varepsilon^2 + A \dt^2 )
  ( \| \Delta_N e^{n+1} \|_2^2 - \| \Delta_N e^n \|_2^2  )  \nonumber
\\
  &\le&
  ( \frac{147}{2} \varepsilon^{-2}  + 1 ) \| e^{n+1} \|_2^2
  + ( 96 \varepsilon^{-2}  + \frac12 ) \| e^n \|_2^2
  + \frac{243}{8} \varepsilon^{-2}  \| e^{n-1} \|_2^2
  + \frac{27}{8} \varepsilon^{-2}  \| e^{n-2} \|_2^2   \nonumber
\\
  &&
  + \frac{\varepsilon^2}{6} ( \|  \Delta_N e^n \|_2^2 + \|  \Delta_N e^{n-1} \|_2^2
  + \|  \Delta_N e^{n-2} \|_2^2 )    + \frac32 \| \tau^n \|_2^2  .
   \label{convergence-6-6}
\end{eqnarray}
By introducing a modified quantity
\begin{eqnarray}
   \widetilde{F}^{n+1} := F^{n+1}
  + \frac12 ( \varepsilon^2 + A \dt^2 ) \dt \| \Delta_N e^{n+1} \|_2^2 ,
   \label{convergence-7-1}
\end{eqnarray}
and making use of the following obvious fact:
\begin{eqnarray}
  \| e^k \|_2^2 \le \frac{1}{\alpha_1^2} \widetilde{F}^k ,  \quad
  \forall k \ge 0 , \label{convergence-7-2}
\end{eqnarray}
we obtain the following estimate
\begin{eqnarray}
  &&
    \frac{1}{\dt} ( \widetilde{F}^{n+1} - \widetilde{F}^n  )
    + \frac56 \varepsilon^2 \| \Delta_N e^{n+1} \|_2^2
    - \frac{\varepsilon^2}{6} ( \|  \Delta_N e^n \|_2^2 + \|  \Delta_N e^{n-1} \|_2^2
  + \|  \Delta_N e^{n-2} \|_2^2 )     \nonumber
\\
  &\le&
  ( 96 \varepsilon^{-2}  + 1 ) \alpha_1^{-2} ( \widetilde{F}^{n+1}
  + \widetilde{F}^n  + \widetilde{F}^{n-1} + \widetilde{F}^{n-2} )
   + \frac32 \| \tau^n \|_2^2  .
   \label{convergence-7-3}
\end{eqnarray}
In turn, an application of discrete Gronwall inequality results in the convergence estimate:
\begin{eqnarray}
   F^{n+1} + \Bigl( \frac13 \varepsilon^2 \dt \sum_{m=1}^{n+1} \| \Delta_N e^m \|_2^2 \Bigr)^{1/2} \le \hat{C} ( \dt^3 + h^m)^2 .
   \label{convergence-7-4}
\end{eqnarray}
Furthermore, its combination with definition~\eqref{defi-F} (for $F^{n+1}$) indicates the desired result~\eqref{convergence-0-2}. This completes the proof of Theorem~\ref{thm:convergence}.

\begin{rem}
The proposed numerical scheme~\eqref{scheme-BDF3-0} is fully discrete, instead of method of line; both the modified BDF3 temporal method and the Fourier pseudo-spectral spatial approximation have been involved in the scheme. In turn, the truncation error $\tau^n$ (appearing in \eqref{BDF3-consistency-1}) implies both the third order temporal accuracy and the spectral spatial accuracy. In the stability and convergence estimates, it is noticed that the summation by parts formula~\eqref{spectral-coll-inner product-3} has played an important role in the theoretical analysis for the fully discrete scheme.
\end{rem}

\begin{rem}
The inner product associated with the truncation error is analyzed in~\eqref{convergence-2}, with $\| \tau^n \|_2 = O (\dt^3 + h^m)$. On the other hand, this accuracy order does not affect the final convergence rate in time, since the numerical error evolutionary equation~\eqref{BDF3-consistency-2} could be rewritten as
\begin{eqnarray}
  &&
  \frac{11}{6} e^{n+1} - 3 e^n + \frac32 e^{n-1} - \frac13 e^{n-2}
  + \varepsilon^2 \dt \Delta_N^2 e^{n+1} +A \dt^3 \Delta_N^2( e^{n+1} - e^n) \nonumber
\\
  &=&
   - \dt \nabla_N \cdot \Bigl( 3 \frac{\nabla_N U^n}{1+ |\nabla_N U^n|^2}
  - 3 \frac{\nabla_N u^n}{1+ |\nabla_N u^n|^2}
  -3 \frac{\nabla_N U^{n-1}}{1+ |\nabla_N U^{n-1}|^2}
  +3 \frac{\nabla_N u^{n-1}}{1+ |\nabla_N u^{n-1}|^2}    \nonumber
\\
  &&
  + \frac{\nabla_N U^{n-2}}{1+ |\nabla_N U^{n-2}|^2}
  - \frac{\nabla_N u^{n-2}}{1+ |\nabla_N u^{n-2}|^2}  \Bigr)
  + \dt \tau^n ,
  \label{BDF3-consistency-2-rewritten}
\end{eqnarray}
so that the last term (corresponding to the truncation error) is of order $O (\dt^4 + \dt h^m)$. Finally, with an application of discrete Gronwall inequality to the preliminary estimate~\eqref{convergence-7-3}, we are able to derive the desired convergence result~\eqref{convergence-7-4}, with third order temporal accuracy and spectral order spatial accuracy.
\end{rem}

\begin{rem}
For the proposed third order numerical scheme~\eqref{scheme-BDF3-0}, a uniform time step is necessary to establish the energy stability and convergence analyses at a theoretical level, due to its multi-step and long-stencil nature. Meanwhile, non-uniform time step size could also be used in the practical computations, and such a choice does not affect the energy dissipation property in the simulation results; see the detailed descriptions in the next section. In addition, there have been some existing works of theoretical analysis of variable time stepping method for gradient flows, such as~\cite{chen19b} for the BDF2 scheme applied to the Cahn-Hilliard equation. The variable time stepping version of the proposed third order numerical scheme~\eqref{scheme-BDF3-0} will also be investigated in the future works.
\end{rem}

\section{Numerical results} \label{sec:numerical results}

\subsection{Convergence test for the numerical scheme}

In this subsection we perform a numerical accuracy check for the third order accurate BDF-type scheme~\eqref{scheme-BDF3-0}. The computational domain is set to be $\Omega = (0,1)^2$, and the exact profile for the phase variable is set to be
	\begin{equation}
U (x,y,t) = \sin( 2 \pi x) \cos(2 \pi y) \cos( t) .
	\label{AC-1}
	\end{equation}
To make $U$ satisfy the original PDE \eqref{equation-NSS}, we have to add an artificial, time-dependent forcing term. Then the proposed third order BDF-type  scheme~\eqref{scheme-BDF3-0} can be implemented to solve for (\ref{equation-NSS}).

To investigate the accuracy in space, we fix $\dt = 10^{-4}$, 
so that the spatial approximation error dominates the overall numerical error. We compute solutions with grid sizes $N = 64$ to $N=144$ in increments of 8, and we solve up to time $T = 1$. The surface diffusion parameter is taken as $\varepsilon=0.05$, and we set the artificial regularization parameter as $A=1$. The $\ell^2$ and $\ell^\infty$ numerical errors, computed by the proposed numerical schemes~\eqref{scheme-BDF3-0}, are displayed in Figure~\ref{fig0}. The spatial spectral accuracy is apparently observed for the phase variable. Due to the round-off errors, a saturation of spectral accuracy appears with an increasing $N$.

	\begin{figure}
	\begin{center}
\includegraphics[width=3.0in]{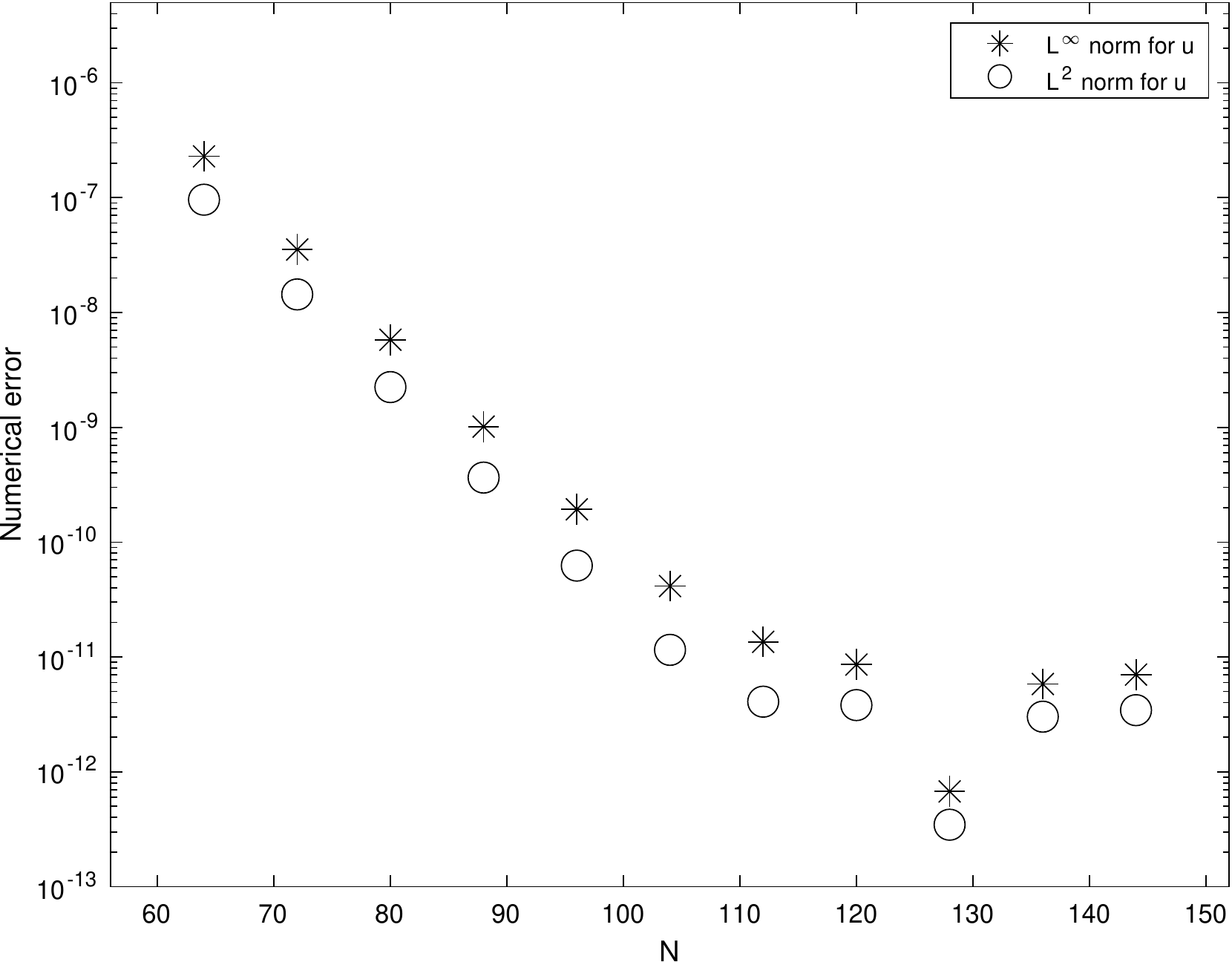}
	\end{center}
\caption{The discrete $\ell^2$ and $\ell^\infty$ numerical errors for the phase variable at $T=1$, plotted versus $N$, the number of spatial grid point, for the fully discrete numerical scheme~\eqref{scheme-BDF3-0}. The time step size is fixed as $\dt = 10^{-4}$. An apparent spatial spectral accuracy is observed.}
	\label{fig0}
	\end{figure}

To demonstrate the accuracy in time, the spatial numerical error has to be negligible. We fix the spatial resolution as $N=192$ (so that $h=\frac{1}{192}$), and set the final time $T=1$. The same surface diffusion and artificial regularization parameters are taken, namely, $\varepsilon=0.05$, $A=1$, respectively. Naturally, a sequence of time step sizes are taken as $\dt=\frac T{N_T}$, with $N_T = 100:100:1000$.  The expected temporal numerical accuracy assumption $e=C \dt^k$ indicates that $\ln |e|=\ln (CT^k) - k \ln N_T$, so that we plot $\ln |e|$ vs. $\ln N_T$ to demonstrate the temporal convergence order. The fitted line displayed in Figure~\ref{fig1} shows an approximate slope of -3, which in turn verifies a nice third order temporal convergence order, in both the discrete $\ell^2$ and $\ell^\infty$ norms. 


	\begin{figure}
	\begin{center}
\includegraphics[width=3.0in]{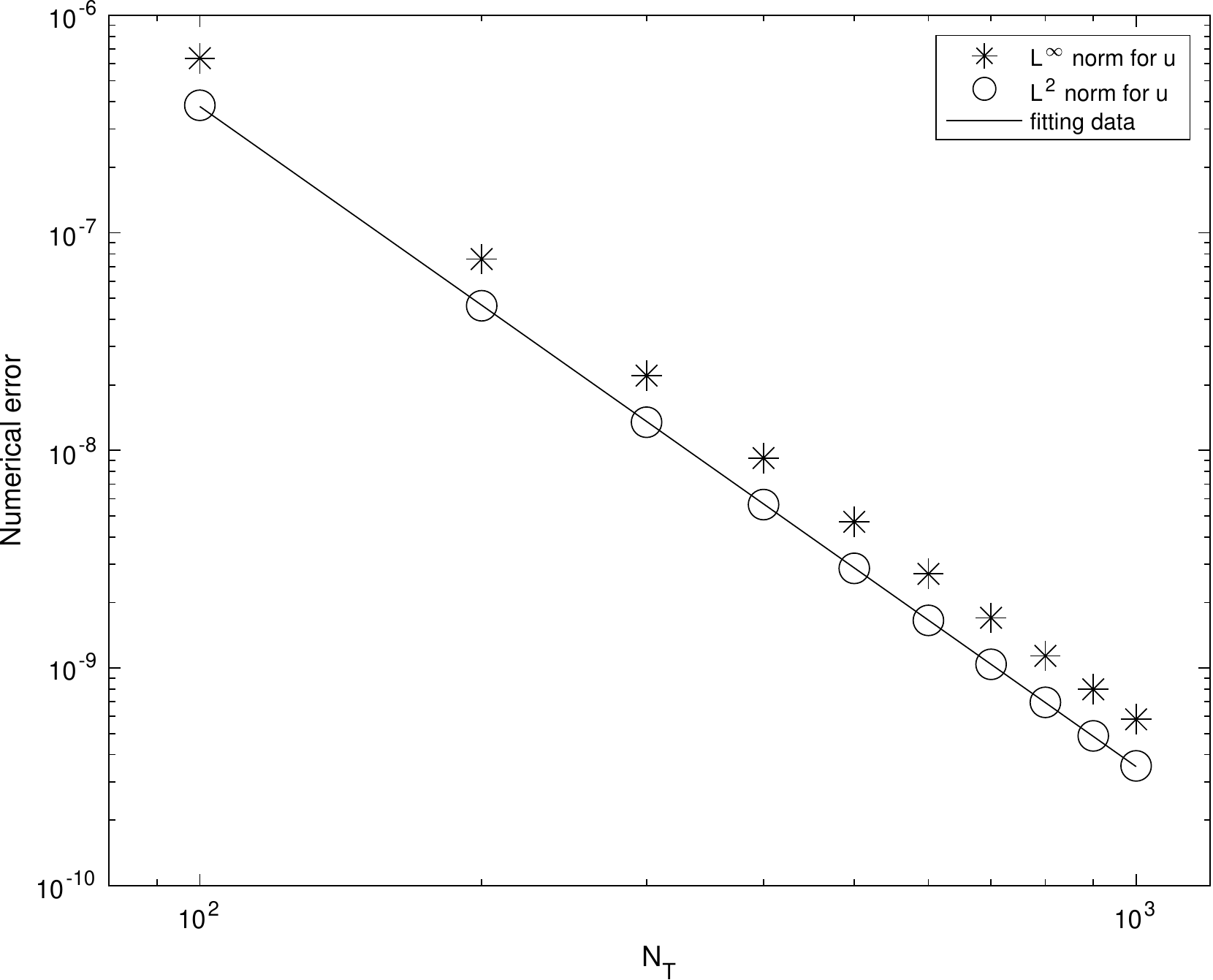}
	\end{center}
\caption{The discrete $\ell^2$ and $\ell^\infty$ numerical errors vs. temporal resolution  $N_T$ for $N_T = 100:100:1000$, with a spatial resolution $N=192$. The surface diffusion parameter is taken to be $\varepsilon=0.05$. The data lie roughly on curves $CN_T^{-3}$, for appropriate choices of $C$, confirming the full third-order accuracy of the scheme.}
	\label{fig1}
	\end{figure}
	
\subsection{Long time simulation results of the coarsening process}

With the assumption that $\varepsilon\ll\min\left\{L_x,L_y\right\}$, the temporal evolution of the solution to \eqref{equation-NSS} has always been of great interests. The physically interesting quantities that may be obtained from the solutions are (i) the energy $E(t)$; (ii) the characteristic (average) height (the surface roughness) $h(t)$;  and (iii) the characteristic (average) slope $m(t)$. These quantities are precisely defined as
	\begin{eqnarray}
h(t) &=&  \sqrt{ \frac{1}{| \Omega|} \int_{\Omega} \Bigl| u ( {\bf x}, t )  - \bar{u} (t) \Bigr|^2 \mathrm{d} {\bf x} }  \ ,  \quad  \mbox{with} \quad \,  \bar{u} (t) :=  \frac{1}{| \Omega|} \int_{\Omega}  u ( {\bf x}, t )  \mathrm{d} {\bf x} ,
 	\label{standard deviation}
	\\
m(t) &=& \sqrt{ \frac{1}{| \Omega|} \int_{\Omega}  \left| \nabla u ( {\bf x}, t ) \right|^2 \mathrm{d} {\bf x} } .
	\label{mound width}
	\end{eqnarray}
For the no-slope-selection equation~\eqref{equation-NSS}, asymptotic scaling law could be formally derived as $h\sim  O\left(t^{1/2}\right)$, $m(t) \sim O\left(t^{1/4}\right)$, and $E\sim O\left(-\ln(t)\right)$ as $t\to\infty$; see~\cite{golubovic97, libo03, libo04} and other related references. This in turn implies that the characteristic (average) length $\ell(t) := h(t)/m(t)  \sim O\left( t^{1/4}\right)$ as $t\to\infty$. In other words, the average length and average slope scale the same with increasing time.  We also observe that the average mound height $h(t)$ grows faster than the average length $\ell(t)$, which is expected because there is no preferred slope of the height function $u$.

At a theoretical level, as described in~\cite{kohn06, kohn03, libo04}, one can at best  obtain lower bounds for the energy dissipation and, conversely, upper bounds for the average height.  However, the rates quoted as the upper or lower bounds are typically observed for the averaged values of the quantities of interest.  It is quite challenging to numerically predict these scaling laws, since very long time scale simulations are needed. To capture the full range of coarsening behaviors, numerical simulations for the coarsening process require short-time and long-time accuracy and stability, in addition to high spatial accuracy for small values of $\varepsilon$.	

In this section we display the numerical simulation results obtained from the proposed third order BDF-type scheme~\eqref{scheme-BDF3-0} for the no-slope-selection equation~\eqref{equation-NSS}, and compare the computed solutions against the predicted coarsening rates.  Similar results have also been reported for many first and second order accurate numerical schemes in the existing literature, such as the ones given by~\cite{chen12, Ju18, wang10}, etc. The surface diffusion coefficient parameter is taken to be $\varepsilon=0.02$ in this article, and the domain is set as $L=L_x = L_y = 12.8$. The uniform spatial resolution is given by $h = L /N$, $N=512$, which is adequate to resolve the small structures in the solution with such a value of $\varepsilon$. The artificial regularization parameter is taken as $A=0.5$ in our simulation. 

For the temporal step size $\dt$, we use increasing values of $\dt$, namely, $\dt = 0.004$ on the time interval $[0,400]$, $\dt = 0.04$ on the time interval $[400,6000]$, $\dt=0.16$ on the time interval $[6000, 3 \times 10^5]$. Whenever a new time step size is applied, we initiate the two-step numerical scheme by  taking $u^{-1} = u^{-2} = u^0$, with the initial data $u^0$ given by the final time output of the last time period. Figure~\ref{fig3} displays time snapshots of the film height $u$ with $\varepsilon=0.02$, with significant coarsening observed in the system.  At early times many small hills (red) and valleys (blue) are present.  At the final time, $t= 300000$, a one-hill-one-valley structure emerges, and further coarsening is not possible.

\begin{figure}[h]
	\begin{center}
		\begin{subfigure}{0.48\textwidth}
			\includegraphics[height=0.48\textwidth,width=0.48\textwidth]{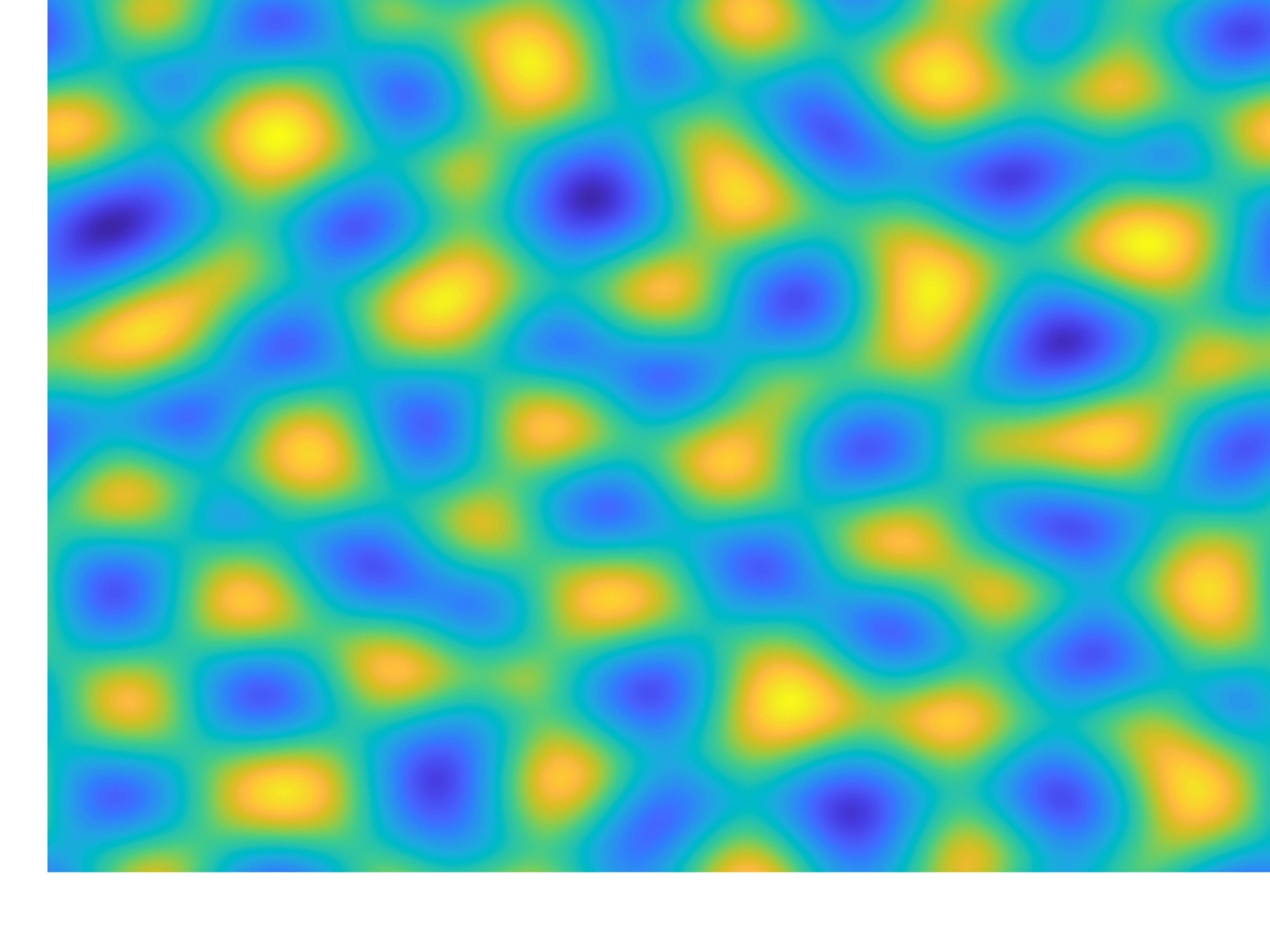}
			\includegraphics[height=0.48\textwidth,width=0.48\textwidth]{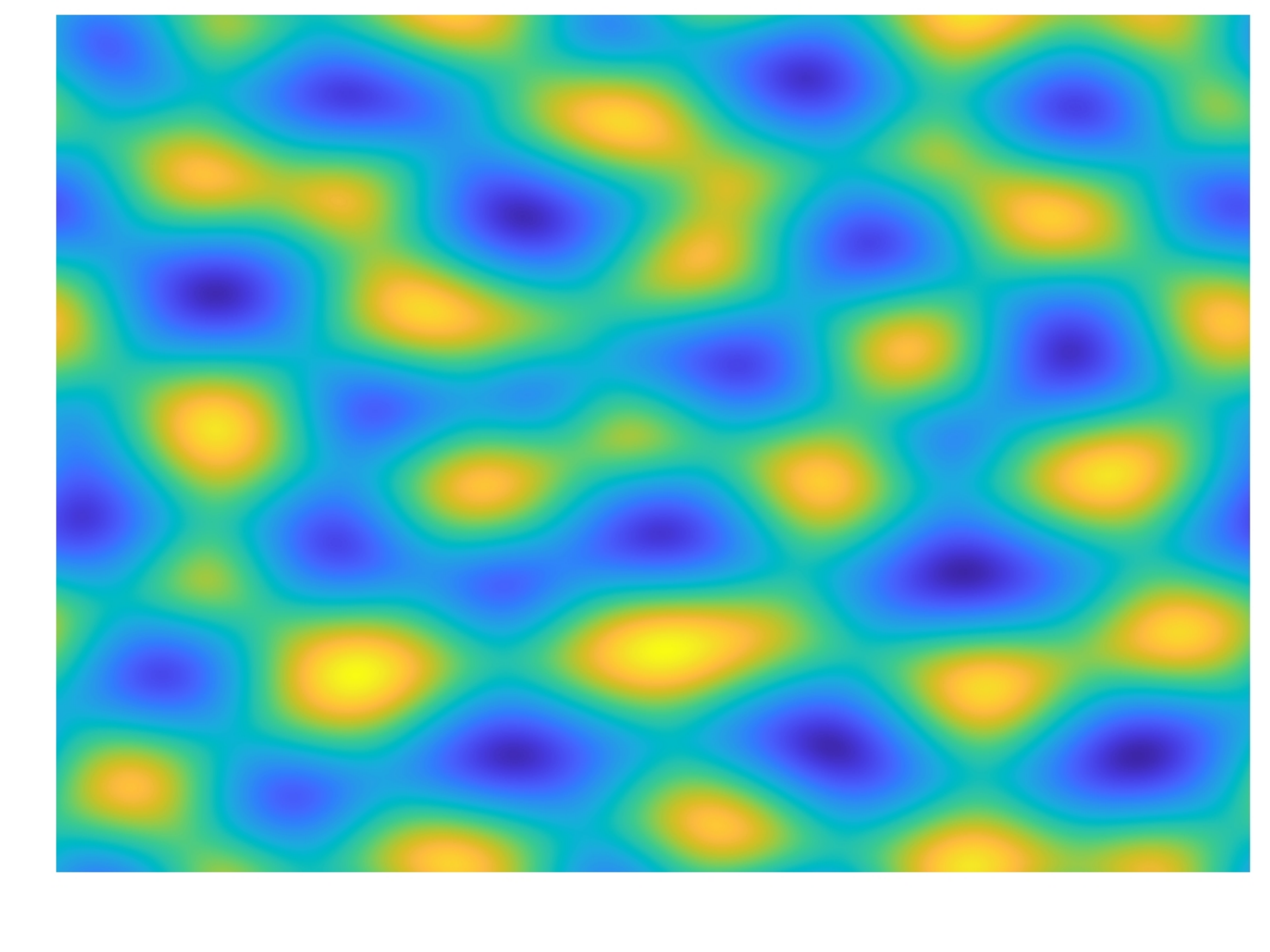}
			\caption*{$t=200, 400$}
		\end{subfigure}
		\begin{subfigure}{0.48\textwidth}
			\includegraphics[height=0.48\textwidth,width=0.48\textwidth]{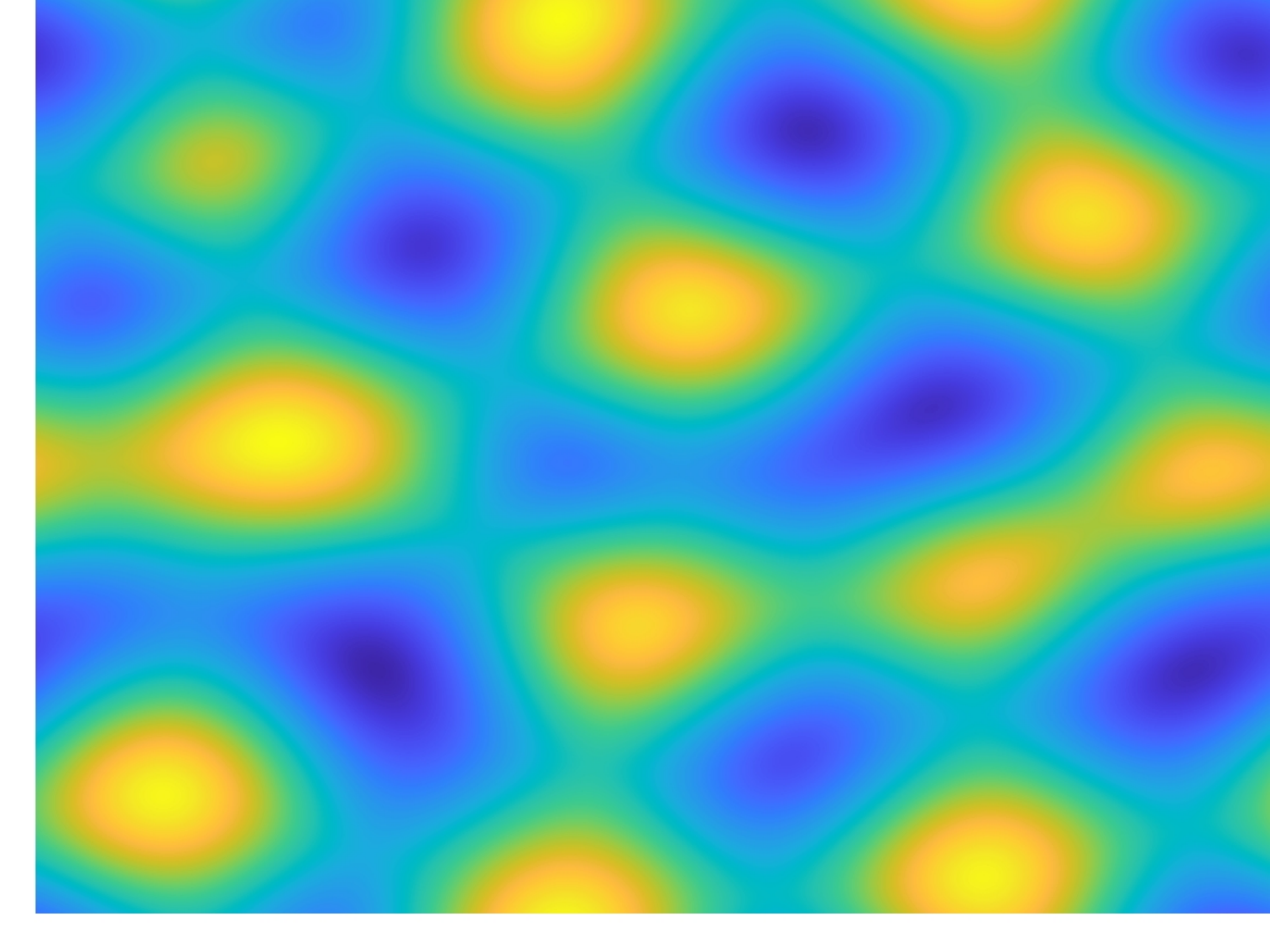}
			\includegraphics[height=0.48\textwidth,width=0.48\textwidth]{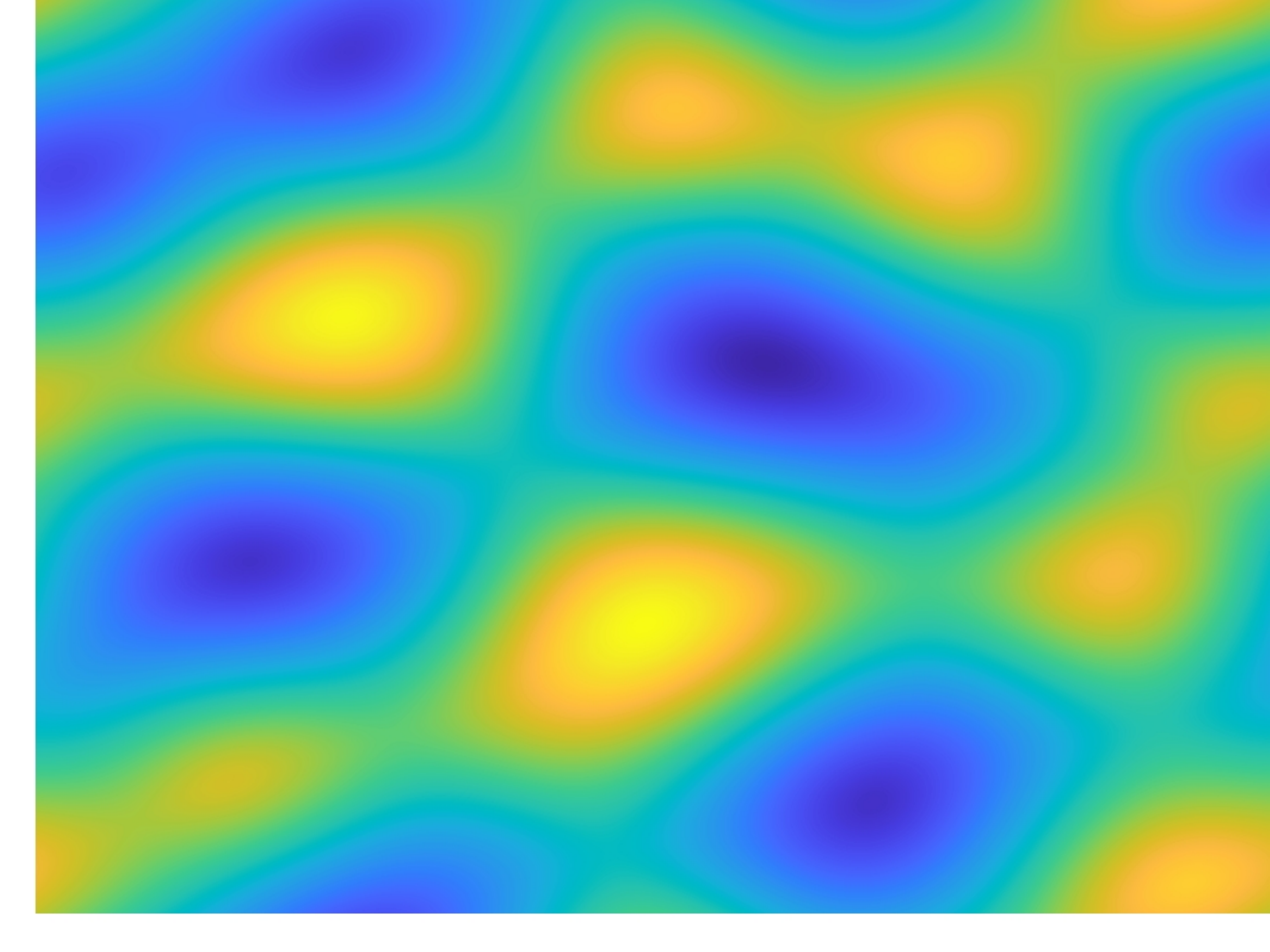}
			\caption*{$t=3000, 6000$}
		\end{subfigure}
		\begin{subfigure}{0.48\textwidth}
			\includegraphics[height=0.48\textwidth,width=0.48\textwidth]{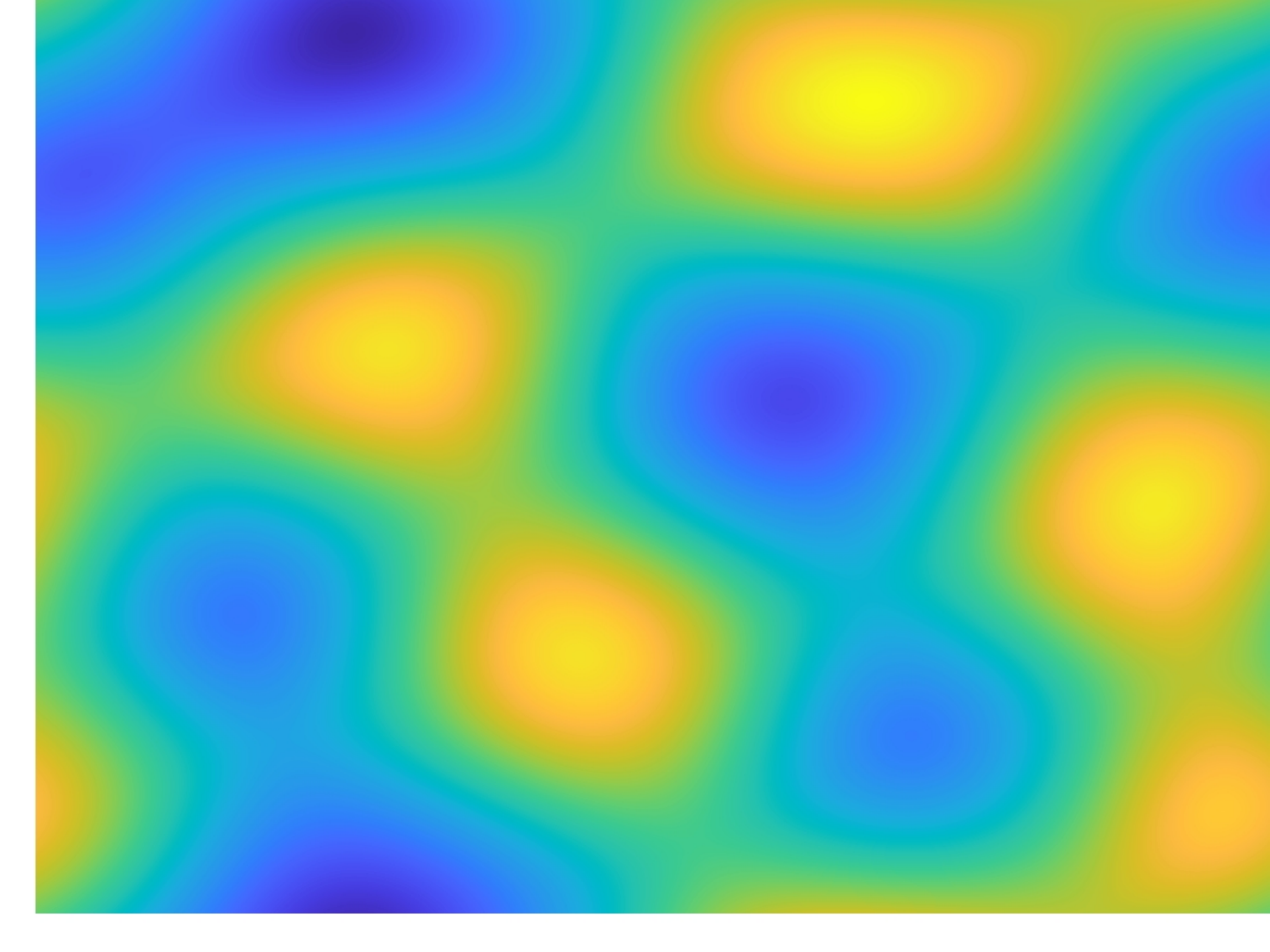}
			\includegraphics[height=0.48\textwidth,width=0.48\textwidth]{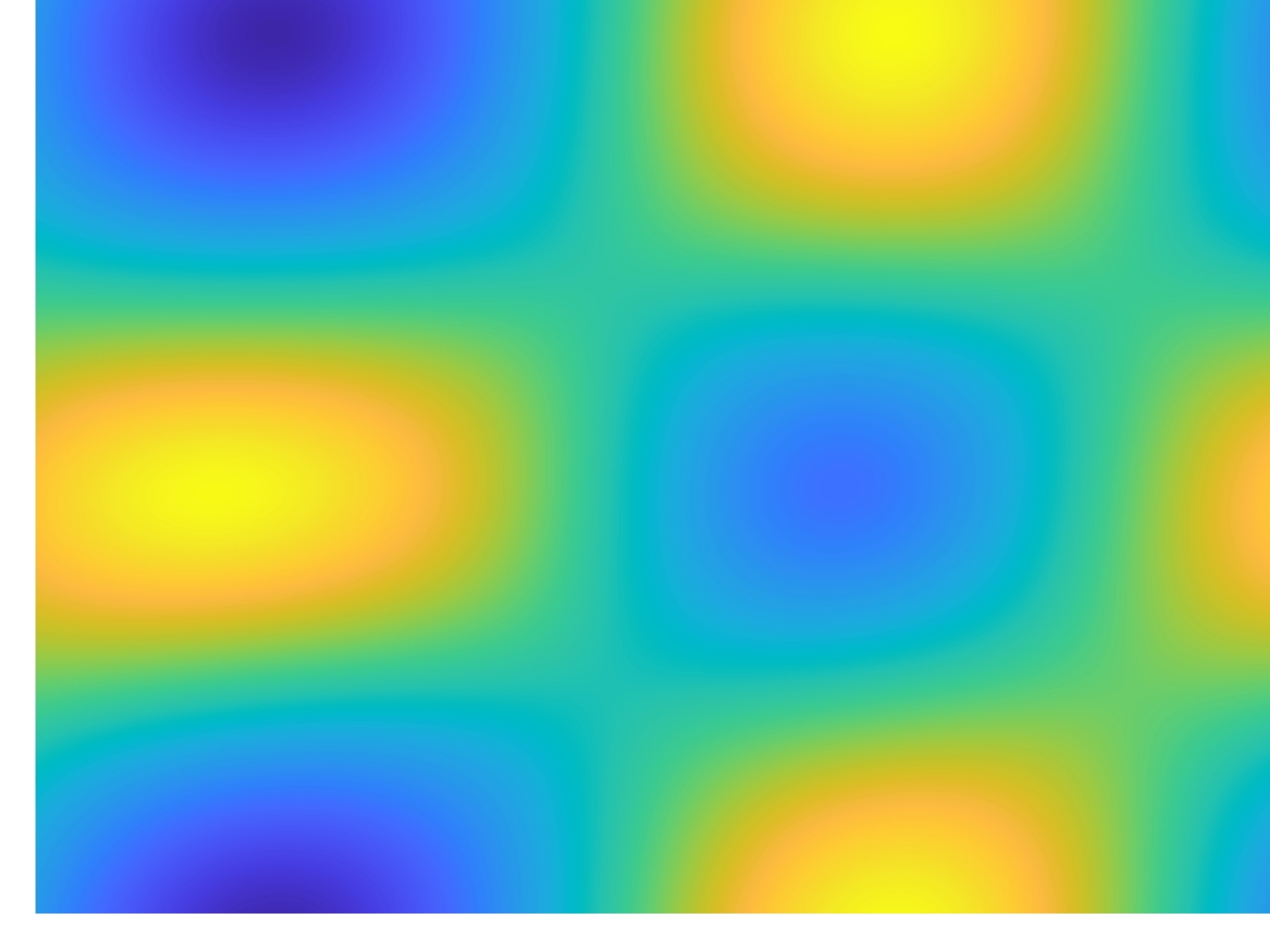}
			\caption*{$t=20000, 40000$}
		\end{subfigure}
		\begin{subfigure}{0.48\textwidth}
			\includegraphics[height=0.48\textwidth,width=0.48\textwidth]{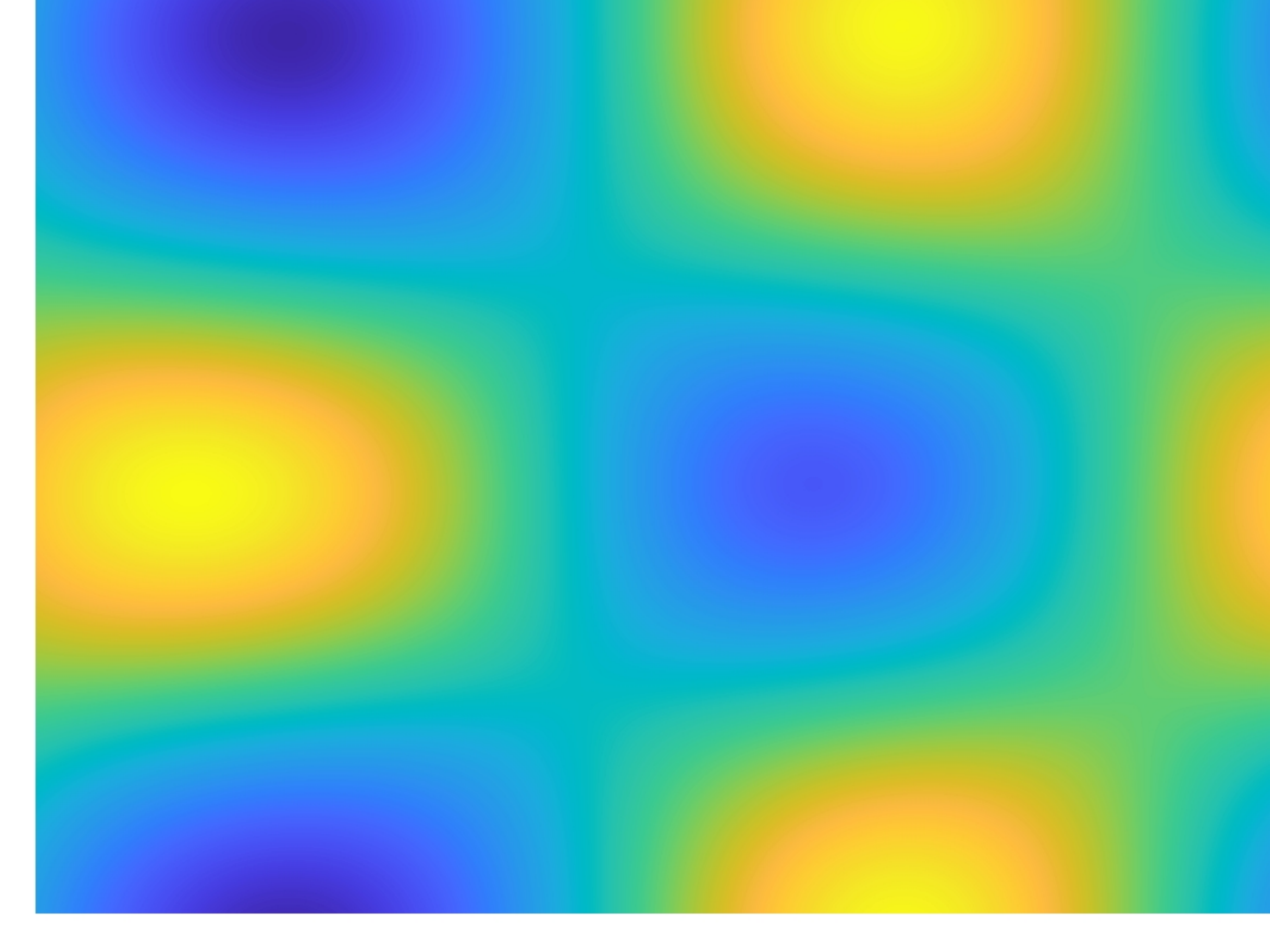}
			\includegraphics[height=0.48\textwidth,width=0.48\textwidth]{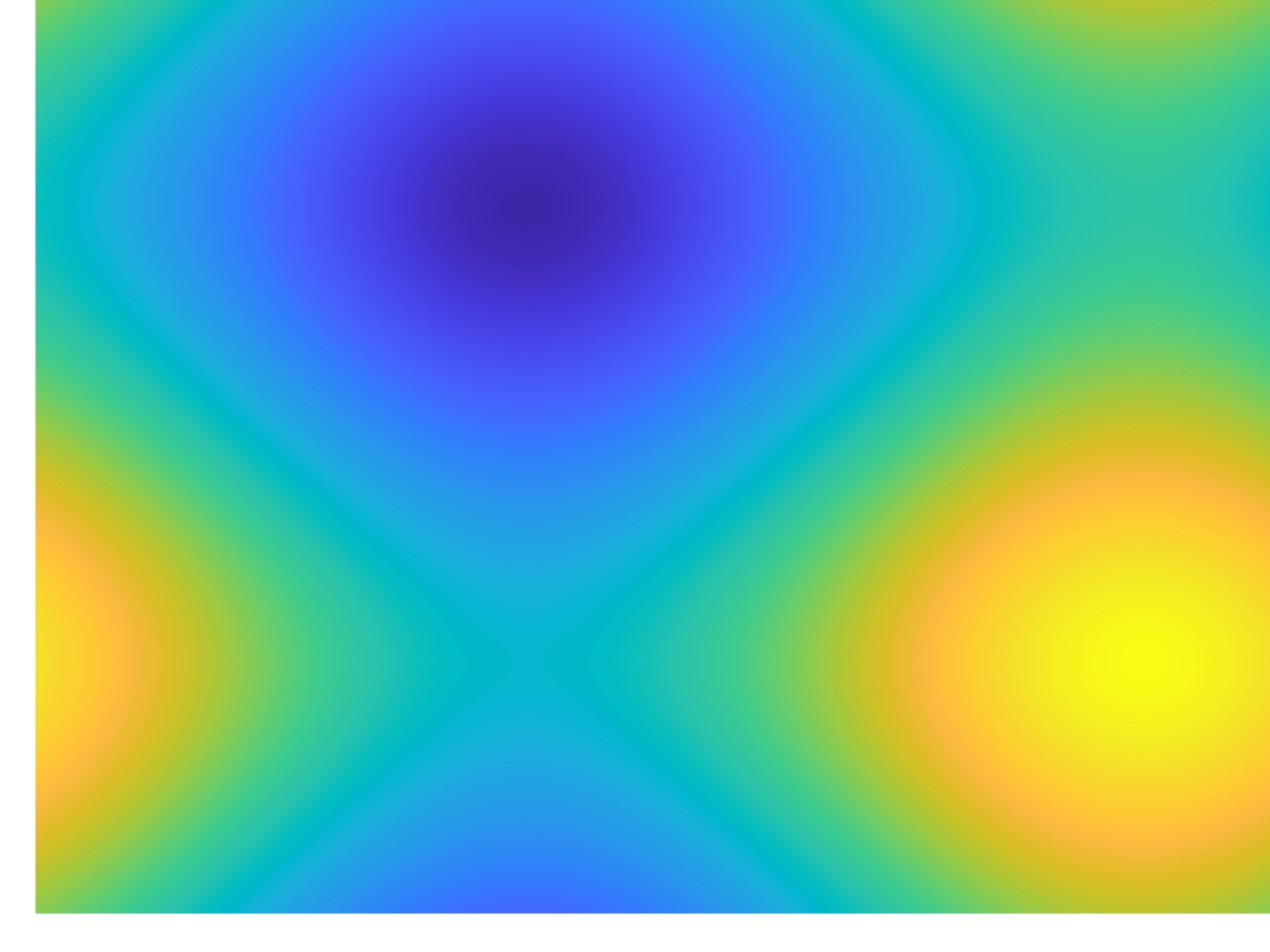}
			\caption*{$t=80000, 300000$}
		\end{subfigure}
\caption{(Color online.) Snapshots of the computed height function $u$ at the indicated times for the parameters $L =12.8$, $\varepsilon = 0.02$.  
The hills at early times are not as high as time at later times, and similarly with the valley. The average height/depth evolution with time could be seen in Figure~\ref{fig4}.}		
		\label{fig3}
	\end{center}
\end{figure}


The long time characteristics of the solution, especially the energy decay rate, average height growth rate, and the mound width growth rate, are of interest to surface physics community.  The last two quantities can be easily measured experimentally. On the other hand, the discrete energy $E_N$ is defined via~\eqref{energy-discrete-spectral}; the space-continuous average height and average slope have been defined in \eqref{standard deviation}, \eqref{mound width}, and the analogous discrete versions are also available. 
Theoretically speaking, the lower bound for the energy decay rate is of the order of $- \ln(t)$, and the upper bounds for the average height  and average slope/average length are of the order of $t^{1/2}$, $t^{1/4}$, respectively, as established for the no-slope-selection equation~\eqref{equation-NSS} in~\cite{libo04}. Figures~\ref{fig4}-\ref{fig6} present the semi-log plots for the energy versus time, log-log plots for the average height versus time, and average slope versus time, respectively, with the given physical parameter $\varepsilon=0.02$. The detailed scaling ``exponents" are obtained using least squares fits of the computed data up to time $t=400$.  A clear observation of the $- \ln(t)$, $t^{1/2}$ and $t^{1/4}$ scaling laws can be made, with different coefficients dependent upon $\varepsilon$, or, equivalently, the domain size, $L$.

	\begin{figure}
	\begin{center}
\includegraphics[width=5.0in]{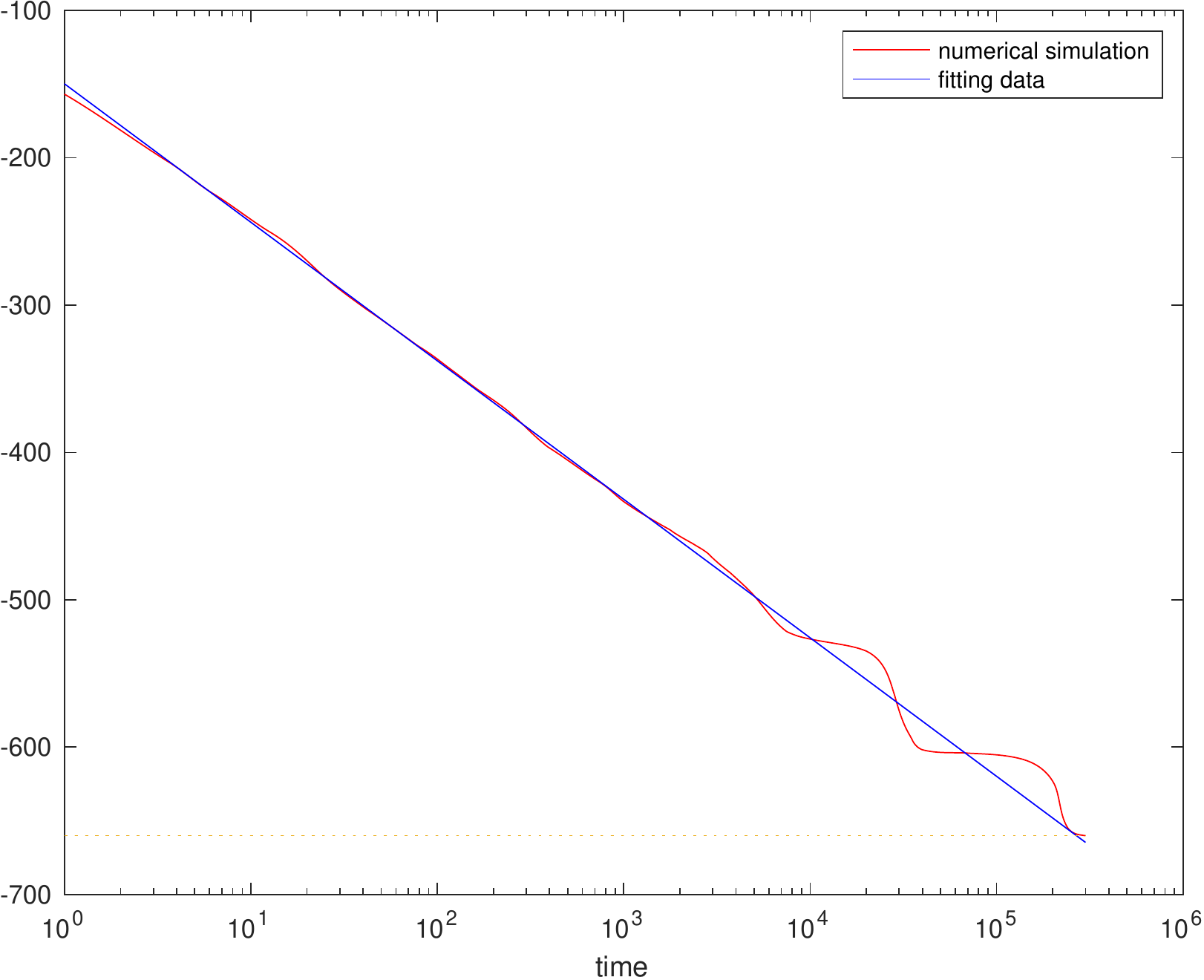}
	\end{center}
\caption{Semi-log plot of the temporal evolution the energy $E_N$ for $\varepsilon=0.02$.  The energy decreases like $-\ln(t)$ until saturation.
The dotted lines correspond to the minimum energy reached by the numerical simulation. The red lines represent the energy plot obtained by the simulations, while the straight lines are obtained by least squares approximations to the energy data.  The least squares fit is only taken for the linear part of the calculated data, only up to about time $t=400$.  The fitted line for the energy has the form $a_e\ln(t)+b_e$, with $a_e = -40.8189$, $b_e=-149.8528$.}
	\label{fig4}
       \end{figure}

	\begin{figure}
	\begin{center}
\includegraphics[width=5.0in]{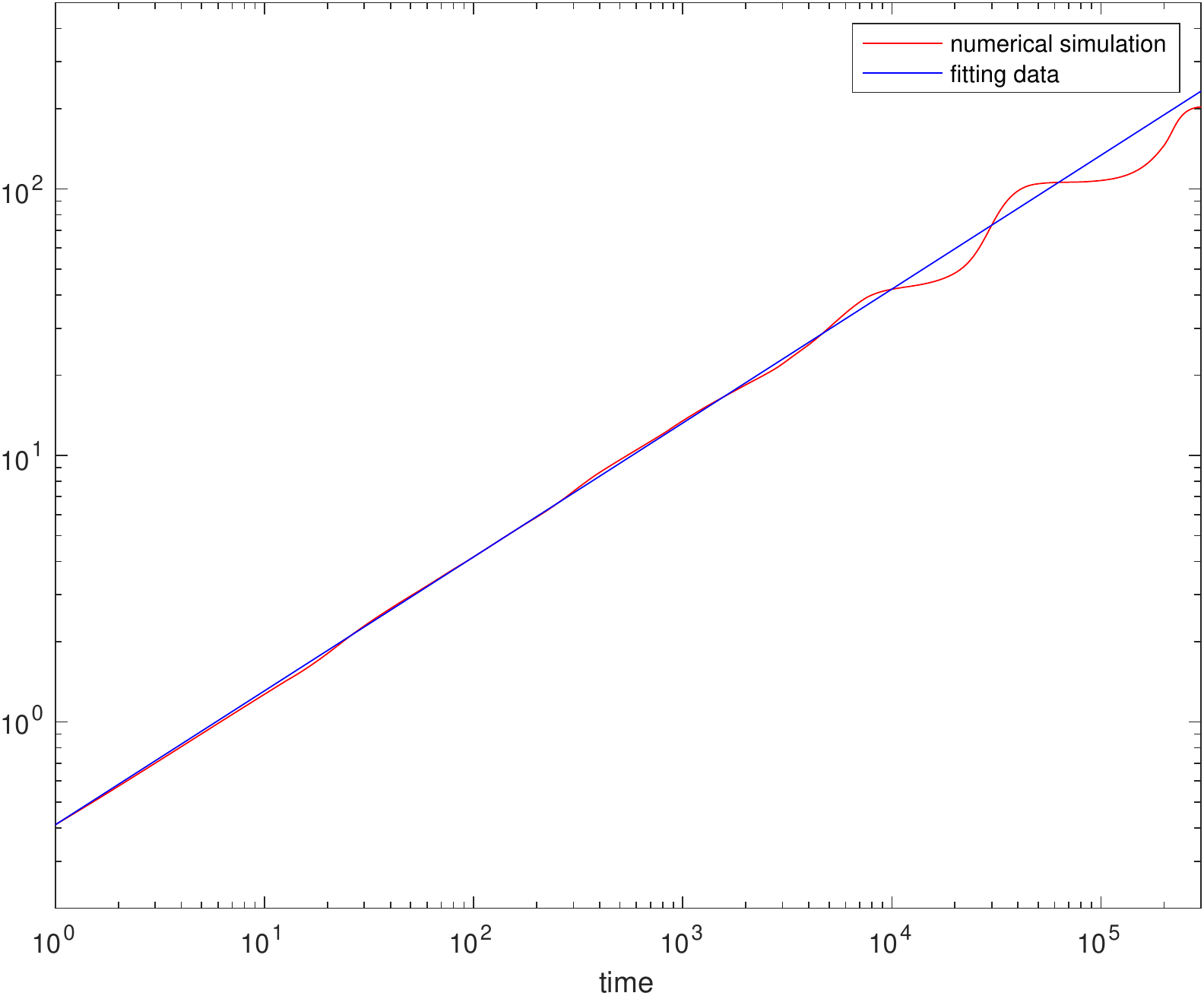}
	\end{center}
\caption{The log-log plot of the average height (or roughness) of $u$, denoted as $h(t)$, which grows like $t^{1/2}$. 
The red lines represent the average height plot obtained by the simulations, while the straight lines are obtained by least squares approximations to the numerical data.  The least squares fit is only taken for the linear part of the calculated data, only up to $t=400$.  The fitted line for the average height has the form $a_ht^{b_h}$, with  $a_h =   0.4113$, $b_h = 0.5025$.}
	\label{fig5}
       \end{figure}

	\begin{figure}
	\begin{center}
\includegraphics[width=5.0in]{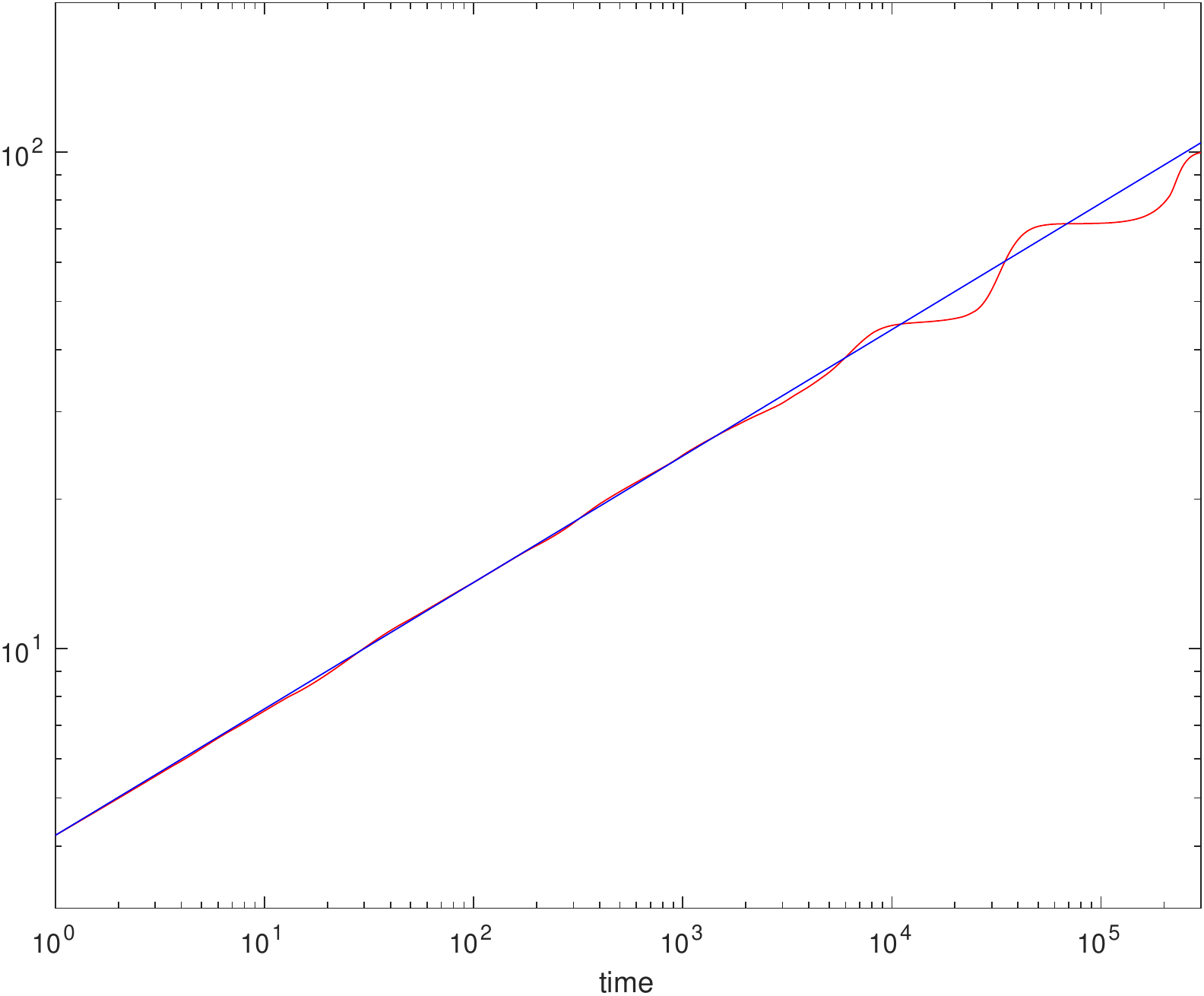}
	\end{center}
\caption{The log-log plot of the average width of $u$, denoted $m(t)$, which grows like $t^{1/4}$. The red lines represent the average width plot obtained by the simulations, while the straight lines are obtained by least squares approximations to the numerical data. Similarly, the least squares fit is only taken for the linear part of the calculated data, only up to $t=400$.  The fitted line for the average width has the form $a_m t^{b_m}$, with $a_m = 4.2063$, $b_m = 0.2547$.}
	\label{fig6}
       \end{figure}

We also recall that a lower bound for the energy (\ref{energy-NSS}), assuming $\Omega = (0,L)^2$, which has been derived and polished in our earlier works~\cite{chen12, chen14, wang10}:
	\begin{equation}
E (\phi) \ge \frac{L^2}{2}\left( \ln\left(\frac{4 \varepsilon^2\pi^2}{L^2}\right)-\frac{4\varepsilon^2\pi^2}{L^2}+1\right) =:\gamma  \ .
	\label{lower-bound}
	\end{equation}
Since the energy is bounded below, it cannot keep decreasing at the rate $-\ln(t)$.  This fact manifests itself in the calculated data as the rate of decrease of the energy, for example, begins to wildly deviate from the predicted $-\ln(t)$ curve.  Sometimes the rate of decrease increases, and sometimes it slows as the systems ``feels" the periodic boundary conditions. In fact, regardless of this later-time deviation from the accepted rates, the time at which the system saturates (\emph{i.e.}, the time when the energy abruptly and essentially stops decreasing) is roughly that predicted by extending the blue lines in Figure~\ref{fig4} to the predicted minimum energy~\eqref{lower-bound}.

	\begin{rem}
In this presented numerical simulation, the spatial resolution and time step sizes are taken as the same as the ones presented for the second order energy stable scheme~\cite{chen14}. 
For the long time simulation, both numerical schemes have produced similar evolutionary curves in terms of energy, standard deviation, and the mound width. A more detailed calculation shows that long time asymptotic growth rate of the standard deviation given by the third order numerical simulation is closer to $t^{1/2}$ than that by the second order energy stable scheme: $m_r = 0.5025$, as recorded in Figure~\ref{fig5}, while in~\cite{chen14} this exponent was found to be $m_r = 0.5132$. Similarly, the long time asymptotic growth rate of the mound width given by~\eqref{scheme-BDF3-0} is closer to $t^{1/4}$ than that by the second order energy stable scheme in: $b_m = 0.2547$, as recorded in Figure~\ref{fig6}, in comparison with $m_r = 0.2607$ reported in~\cite{chen14}. This gives more evidence that the third order BDF-type scheme is able to produce more accurate long time numerical simulation results than the second order schemes. 

Similar comparison has also been made between the second order ETD-related scheme and the proposed third order BDF-type scheme: $m_r = 0.5025$, $b_m=0.2547$ for the proposed third order scheme, in comparison with $m_r=0.510$, $b_m=0.258$, for the second order scheme as reported in~\cite{Ju18}. This gives another evidence of robustness of the third order accurate numerical scheme for the NSS equation~\eqref{equation-NSS}.

	\end{rem}

\section{Concluding remarks} \label{sec:conclusion}

In this article, we propose and analyze a third order accurate BDF-type numerical scheme for the NSS equation~\eqref{equation-NSS} of the epitaxial thin film growth model, combined with Fourier pseudo-spectral spatial discretization. The surface diffusion term is treated implicitly, while the nonlinear chemical potential is approximated by a third order explicit extrapolation formula for the sake of solvability. More importantly, a third order accurate Douglas-Dupont regularization term is added in the numerical scheme. The energy stability is derived in a modified version, based on a careful energy stability estimate, combined with Fourier eigenvalue analysis; a theoretical justification of the coefficient $A$ has also been provided. Such an energy stability analysis implies a uniform in time bound of the numerical energy. Furthermore, the optimal rate convergence analysis and error estimate are derived in details. In particular, a discrete inner product is taken with an alternate numerical term, to avoid the well-known difficulty associated with the long-stencil nature of the standard BDF3 scheme.  Some numerical simulation results are presented to demonstrate the robustness of the numerical scheme and the third order convergence. In particular, the long time simulation results have revealed that, the power index for the surface roughness and the mound width growth for $\varepsilon=0.02$ (up to $T=3 \times 10^5$), created by the proposed third order numerical scheme, is more accurate than these created by certain second order accurate, energy stable schemes in the existing literature.

	\section*{Acknowledgements}
This work is supported in part by NSFC 11971047 (Q. Huang) and NSF DMS-2012669 (C.~Wang). 

	\bibliographystyle{plain}
    \bibliography{revision}

	\end{document}